\documentclass[10pt, reqno]{amsart}
\usepackage{amssymb,amsmath,amsthm,color, cite}
\theoremstyle{plain}
\newtheorem{theorem}{Theorem}[section]
\newtheorem{proposition}[theorem]{Proposition}

\newtheorem{lemma}[theorem]{Lemma}
\newtheorem{corollary}[theorem]{Corollary}

\theoremstyle{remark}
\newtheorem{remark}[theorem]{Remark}

\usepackage
[bookmarks=true,
   bookmarksnumbered=false,
   bookmarkstype=toc]
{hyperref}

\numberwithin{equation}{section}

\newcommand{\C}{\mathbb{C}}
\newcommand{\R}{\mathbb{R}}

\newcommand{\N}{\mathcal{N}}

\newcommand{\F}{\mathcal{F}}

\renewcommand{\Im}{\operatorname{Im}}
\renewcommand{\Re}{\operatorname{Re}}

\def\({\left(}
\def\){\right)}

\newcommand{\eps}{\varepsilon}

\DeclareMathOperator{\arctanh}{arctanh}

\newcommand{\qtq}[1]{\quad\text{#1}\quad}

\begin{document}

\title[NLS with a delta potential]{Stability of small solitary waves for the 1$d$ NLS with an attractive delta potential}

\author[S. Masaki]{Satoshi Masaki}
\address{Department of Systems Innovation \\
Graduate School of Engineering Science \\
Toyonaka, Osaka, Japan}
\email{masaki@sigmath.es.osaka-u.ac.jp}

\author[J. Murphy]{Jason Murphy}
\address{Department of Mathematics and Statistics, \\ Missouri University of Science and Technology \\ Rolla, MO, USA}
\email{jason.murphy@mst.edu}

\author[J. Segata]{Jun-ichi Segata}
\address{Mathematical Institute, Tohoku University\\
6-3, Aoba, Aramaki, Aoba-ku, Sendai 980-8578, Japan}
\email{segata@m.tohoku.ac.jp}

\begin{abstract} We consider the initial-value problem for the one-dimensional nonlinear Schr\"odinger equation in the presence of an attractive delta potential.  We show that for sufficiently small initial data, the corresponding global solution decomposes into a small solitary wave plus a radiation term that decays and scatters as $t\to\infty$.  In particular, we establish the asymptotic stability of the family of small solitary waves. 

\end{abstract}

\maketitle

\section{Introduction}

We study the one-dimensional nonlinear Schr\"odinger equation (NLS) with an attractive delta potential.  This equation takes the form
\begin{equation}\label{nls}
\begin{cases}
i\partial_t u = Hu + \mu |u|^p u, \\
u(0)=u_0.
\end{cases}
\end{equation}
Here we take $u:\R_t\times\R_x\to\C$, $\mu\in\R\backslash\{0\}$, and $H$ is the Schr\"odinger operator
\[
H=-\tfrac12\partial_x^2 + q\delta(x),
\]
where $q<0$ (the attractive case) and $\delta$ is the Dirac delta distribution.  Equation \eqref{nls} provides a simple model describing the resonant nonlinear propagation of light through optical wave guides with localized defects \cite{GHW}. For reasons to be detailed below, we consider the $L^2$-supercritical case, namely, $p\geq 4$.  For technical simplicity we also assume $p$ is an even integer. 

In the repulsive case ($q>0$), equation \eqref{nls} is studied from the point of view of scattering. The authors of \cite{BV} proved global well-posedness and scattering in the energy space for the defocusing mass-supercritical case. The work \cite{II} considered the focusing mass-supercritical regime and proved scattering below the ground state threshold. In our previous work \cite{MMS}, we considered \eqref{nls} with a cubic nonlinearity and proved decay and (modified) scattering for small initial data in a weighted space (see also \cite{S}).    

Such results are not expected in the attractive case.  Indeed, in the attractive case the operator $H$ has a single eigenvalue $-\tfrac12 q^2$, with a one-dimensional eigenspace spanned by the $L^2$-normalized eigenfunction
\[
\phi_0(x):=|q|^{\frac12} e^{q|x|}.
\]
One can then prove that there exists a family of small nonlinear bound states $Q$, parametrized by small $z\in\C$, which satisfy
\begin{equation}\label{elliptic}
HQ + \mu|Q|^p Q = EQ,
\end{equation}
with $Q=Q[z]=z\phi_0 + \mathcal{O}(z^2)$ and $E=E[|z|]=-\tfrac12 q^2+\mathcal{O}(z)$. The functions $u(t)=e^{-iEt}Q$ are then small solitary wave solutions to \eqref{nls}.  In particular, one does not expect small solutions simply to decay and scatter in general.  Instead, we will show that for small initial data, the corresponding solution decouples into a small solitary wave plus radiation.   The existence and properties of $Q[z]$ are discussed in Section~\ref{S:existence}.  In fact, in the special case of the delta potential, one can find explicit formulas for the nonlinear ground states.

Our main result is the following theorem.  We write $P_c$ for the projection onto the continuous spectral subspace of $H$. The notation $D_j$ denotes derivative with respect to $z_j$, where we identify $z\in\C$ with the real vector $(z_1,z_2)$.  Finally, $\langle\cdot,\cdot\rangle$ denotes the standard $L^2$ inner product. 

\begin{theorem}\label{T} Let $\|u_0\|_{H^1}=\delta$, $q<0$, and let $p\geq 4$ be an even integer.  For $\delta$ sufficiently small,  there exists a unique global solution $u$ to \eqref{nls} and $z(t)\in\C$ such that writing 
\begin{equation}\label{thm-decomp}
u(t)=Q[z(t)]+v(t),
\end{equation}
where $Q[z(t)]$ is the solution to \eqref{elliptic}, we have the following:
\begin{itemize}
\item $v$ satisfies the orthogonality conditions
\begin{equation}\label{thm-orthogonal}
\Im\langle v(t), D_jQ[z(t)]\rangle \equiv 0 \qtq{for}j\in\{1,2\}.
\end{equation}
\item $v$ obeys the following global space-time bounds, 
\[
\|v\|_{L_t^\infty H_x^1\cap L_t^4 L_x^\infty} + \| \langle x\rangle^{-\frac32}v\|_{L_x^\infty L_t^2} + \|\partial_x v\|_{L_x^\infty L_t^2} \lesssim \delta,
\]
and there exists unique $v_+\in P_c H^1$ such that
\[
\lim_{t\to\infty} \|v(t)-e^{-itH}v_+\|_{H^1}=0.
\]
\item $\|z\|_{L_t^\infty} \lesssim \delta$  and there exists $z_+\in\C$ satisfying $\bigl||z_+|-|z(0)|\bigr|\lesssim \delta^2$ and
\begin{equation}\label{zplus}
\lim_{t\to\infty} z(t)\exp\biggl\{i\int_0^t E[z(s)]\,ds\biggr\} = z_+.
\end{equation}
\end{itemize}
\end{theorem}

Theorem~\ref{T} shows that any small solution decomposes into a nonlinear bound state plus a radiation term.  In particular, we have the asymptotic stability of the family of small solitary waves.  The condition \eqref{thm-orthogonal} makes $v(t)$ orthogonal to the non-decaying solutions of the linearization of \eqref{nls} about the solitary wave at $z(t)$; this is an essential ingredient for establishing decay and scattering for $v$ (see Section~\ref{S:setup} for further discussion).   

Theorem~\ref{T} fits in the context of the stability of small solitary waves for nonlinear Schr\"odinger equations with potential, for which there are many results available. An even more extensive literature exists concerning other notions of stability, stability of large solitary waves, and so on.  We refer the interested reader to \cite{GNT, Mizumachi, Kirr1, Kirr2, SW, SW2, SW3, Weder, TY1, TY2, TY3, TY4, SC1, SC2, SC3} for a sample of the many relevant results that are available.  See in particular \cite{DH, DP, HMZ1, HMZ2, HZ, HZ2} for related results in the setting of NLS with a delta potential. We will keep our focus on the discussion of small solitary waves. 

Our result is closely related to those appearing in \cite{GNT, Mizumachi}, both of which prove asymptotic stability of small solitary waves for NLS with a potential that supports a single negative eigenvalue, with data in $H^1$ and mass-supercritical nonlinearities.  In \cite{GNT}, the authors relied crucially on the endpoint Strichartz estimate in three dimensions.  In \cite{Mizumachi}, T. Mizumachi addressed the one-dimensional case, in which case the usual endpoint Strichartz estimate is unavailable.  His approach was to establish suitable linear estimates in `reversed' Strichartz spaces, in which case the $L_t^2$ endpoint comes back into play. 

Theorem~\ref{T} is an analogue of the main result appearing in \cite{Mizumachi}, which treats a class of potentials that does not include the attractive delta potential.  The key to extending this type of result to the delta potential is to observe that by relying on exact identities related to the Schr\"odinger operator with a delta potential, one can recover the full range of linear estimates that played such an essential role in \cite{Mizumachi}.  We carry this out in Section~\ref{S:strichartz}.  Once the requisite linear estimates are in place, one could then follow many of the remaining arguments in \cite{Mizumachi} rather directly, although this is not the route that we take.  Instead, we set up the problem and prove the main result in a way that that is inspired by the presentation in \cite{GNT}, which we found to be rather conceptually clear. 

Our result is also closely tied to the work of Fukuizumi, Ohta, and Ozawa \cite{FOO}, who studied the focusing $1d$ NLS with an attractive delta potential (see also \cite{GHW}).  These authors considered the problem of stability and instability of nonlinear bound states, relying in particular on explicit formulas that they derived for the nonlinear bound states (see Section~\ref{S:existence} below).  They proved that in the mass-subcritical and mass-critical case, nonlinear bound states are orbitally stable.  In the mass-supercritical case, they show that there exists $E_1<-\tfrac12 q^2$ such that ground states corresponding to $E\in(E_1,-\tfrac12 q^2)$ are orbitally stable, while those corresponding to $E\in(-\infty,E_1)$ are unstable.  Thus our main result, Theorem~\ref{T}, extends the result of \cite{FOO} in the mass-supercritical case to {asymptotic} stability for $E$ in a neighborhood of $-\tfrac12 q^2$.  Furthermore, we are also able to treat the case of a defocusing nonlinearity; we provide explicit formulas for the nonlinear bound states in this case, as well (see Section~\ref{S:existence}).

Finally, we would also like to mention the result of \cite{DP}, which establishes the asymptotic stability of solitons for the focusing cubic NLS with a delta potential and even initial data by making use of complete integrability and the method of nonlinear steepest descent.  This result in particular extended the results appearing \cite{DH, HMZ1, HMZ2, HZ}. 

As mentioned above, our previous work on the $1d$ NLS with a repulsive delta potential \cite{MMS} considered the case of a cubic nonlinearity.  It is an interesting question whether one also has asymptotic stability in the setting of an attractive potential and $L^2$-subcritical nonlinearities (recall that {orbital} stability was proven by \cite{FOO, GHW}). Proving asymptotic stability would most likely require the introduction of stronger integrability conditions on the initial data; for example, this is the case in \cite{Kirr1, Kirr2}, which proved stability of small solitary waves for NLS with potential for some mass-subcritical nonlinearities in dimensions $d\in\{2,3\}$.  In our case, we start only with $H^1$ data and are therefore restricted to $p\geq 4$; this is completely analogous to the situation of trying to prove small-data scattering for the standard power-type NLS.  To see specific the technical points that lead to this restriction, see the estimates of the $|v|^p v$ term in the proofs of Lemma~\ref{L:bs-str}, Lemma~\ref{L:bs-H1}, and Lemma~\ref{L:bs-reverse} (as well as the $\mathcal{O}(v^pQ)$ term in Lemma~\ref{L:bs-H1}). 

Briefly, the proof of Theorem~\ref{T} goes as follows.  One shows that as long as the $u$ remains small in $H^1$, there exists a unique decomposition \eqref{thm-decomp} such that \eqref{thm-orthogonal} holds.  Using \eqref{nls} and differentiating \eqref{thm-orthogonal} leads to a coupled system of equations for $v(t)$ and $z(t)$.  Relying largely on estimates for the linear propagator $e^{-itH}$ and estimates on the bound states $Q[z]$ for small $z$, one can use these equations to close a bootstrap argument, proving that the smallness of $u$ in $H^1$ (as well as the smallness of $v$ and $z$ in various norms) persists.  Thus, one can extend the decomposition for all times; furthermore, the bounds proved on $v$ and $z$ suffice to establish the asymptotics claimed in Theorem~\ref{T}.  The particular choice of the orthogonality condition \eqref{thm-orthogonal} guarantees that the ODE involving $z[t]$ is at least quadratic in $v$, which is essential for proving the necessary bootstrap estimates; see Remark~\ref{R} for further discussion of this point. 

\subsection*{Outline of the paper} In Section~\ref{S:prelim} we introduce notation and gather some preliminary results.  We introduce the linear operator $H$ in Section~\ref{S:linear}.  In Section~\ref{S:strichartz}, we prove a range of Strichartz and local smoothing estimates for $e^{-itH}P_c$.  These match the form of the estimates of Mizumachi \cite{Mizumachi}, who considered a class of potentials that did not include the delta potential.  We are able to give rather direct proofs using the explicit formula for the resolvent.  We also prove a technical result related to the comparison of the $\dot H^1$ inner product to the bilinear form given by $HP_c$.  In Section~\ref{S:existence} we discuss the existence and properties of small nonlinear bound states, and in Section~\ref{S:LWP} we record a local well-posedness result for \eqref{nls}.  In Section~\ref{S:setup} we set up the problem, describing in detail how to find the decomposition \eqref{thm-decomp} satisfying \eqref{thm-orthogonal}.  Finally, in Section~\ref{S:proof} we carry out the main bootstrap argument and complete the proof of Theorem~\ref{T}.

\subsection*{Acknowledgements} S. Masaki was supported by JSPS KAKENHI Grant Numbers 17K14219, 17H02854, and 17H02851.  J. Segata is partially supported by JSPS KAKENHI Grant Number 17H02851.

\section{Preliminaries}\label{S:prelim}

We begin by recording some notation.  We write
\[
\langle f,g\rangle = \int \bar f g\,dx
\]
for the usual $L^2$ inner product.  Throughout the paper we will write $F(u)=\mu|u|^p u$ for the nonlinearity.  We write $\F f$ or $\hat f$ for the Fourier transform.  We write $A\lesssim B$ to denote $A\leq CB$ for some $A,B,C>0$.  

Constants below may depend on the parameter $q$ (the strength of the potential), but we will not make explicit reference to this dependence.  We would like to point out that some of the implicit constants in the estimates for $e^{-itH}P_c$ below would blow up as $|q|\to 0$ (for example, when the proof relies on the fact that $|q-i\mu|\gtrsim |q|$ for $\mu\in\R$).   In particular, the small parameter $\delta$ appearing in the statement of the main result (Theorem~\ref{T}) depends on $q$ and would degenerate to zero as $|q|\to 0$. 

\subsection{Linear theory}\label{S:linear}
The linear Schr\"odinger equation with a delta potential is a classical model in quantum mechanics that is covered extensively in the work \cite{AGH}.  We consider in this paper the case of an attractive delta potential of the form
\[
H=-\tfrac12\partial_x^2 + q\delta(x),\quad q<0.
\]

More precisely, the operator $H$ is defined by $-\tfrac12\partial_x^2$ on its domain
\[
D(H)=\{u\in H^1(\R)\cap H^2(\R\backslash\{0\}):\partial_x u(0+)-\partial_x u(0-)=2qu(0)\}
\]
and extends to a self-adjoint operator on $L^2$ with purely absolutely continuous essential spectrum equal to $[0,\infty)$.  If $q>0$ (the repulsive case) then $H$ has no eigenvalues.  In $q<0$ (the attractive case) then $H$ has a single negative eigenvalue $-\tfrac12 q^2$ with a one-dimensional eigenspace spanned by the $L^2$-normalized eigenfunction
\[
\phi_0(x):=|q|^{\frac12} e^{q|x|}.
\]
In this paper we restrict attention to the attractive case. 

\subsection{Local smoothing and Strichartz estimates}\label{S:strichartz}  In this section we prove several local smoothing and Strichartz estimates for $e^{-itH}$.  We write $P_c$ to denote the projection onto the absolutely continuous spectrum.

The starting point for the estimates we will prove is the following spectral resolution of the free propagator:
\begin{equation}\label{stone}
\begin{aligned}
e^{-itH}P_c & =\int_0^\infty e^{-it\lambda}E(\lambda)\,d\lambda,\\
E(\lambda)&:= \tfrac{1}{2\pi i}[R(\lambda+i0)-R(\lambda-i0)].
\end{aligned}
\end{equation}
Here $R(z)=(H-z)^{-1}$ is the resolvent, and $R(\lambda\pm i0)$ denotes the analytic continuation onto the real line from the upper/lower half plane. For the case of the delta potential, we have explicit formulas for the integral kernel of the resolvent, namely
\begin{align*}
R(\lambda+i0;x,y)& =\tfrac{i}{2\sqrt{\lambda}}\bigl[e^{i|x-y|\sqrt{\lambda}}-\tfrac{q}{q-i\sqrt{\lambda}}e^{i(|x|+|y|)\sqrt{\lambda}}\bigr], \\
R(\lambda-i0;x,y) & = \tfrac{i}{2\sqrt{\lambda}}\bigl[e^{-i|x-y|\sqrt{\lambda}}-\tfrac{q}{q+i\sqrt{\lambda}}e^{-i(|x|+|y|)\sqrt{\lambda}}\bigr]
\end{align*}
for $\lambda>0$.  We similarly write $E(\lambda;x,y)$ for the kernel of $E(\lambda)$.  These identities can be found, for example, in \cite[Chapter~I.3]{AGH}, but they are also readily derived by hand.  In particular, one can recognize the first term as the free resolvent, while the second term (representing the contribution of the potential) simply fixes the boundary condition. 

Typically we will focus on estimating $R(\lambda+i0)$, as the other term is similar.  We write the kernel in two pieces, namely
\[
R(\lambda+i0;x,y)=R_1(\lambda;x,y)+R_2(\lambda;x,y),
\] 
where
\begin{align}
R_1(\lambda;x,y) & = \tfrac{i}{2\sqrt{\lambda}}[e^{i|x-y|\sqrt{\lambda}}-e^{i(|x|+|y|)\sqrt{\lambda}}], \label{Res1} \\
R_2(\lambda;x,y) & =\tfrac{1}{2(q-i\sqrt{\lambda})}e^{i(|x|+|y|)\sqrt{\lambda}}.\label{Res2} 
\end{align}

We note that
\begin{equation}\label{R1}
R_1(\lambda;x,y) = \tfrac{i}{2\sqrt{\lambda}} 
\begin{cases} [e^{-ix\sqrt{\lambda}}-e^{ix\sqrt{\lambda}}]e^{iy\sqrt{\lambda}} & y\geq x \geq 0, \\
0 & y\geq 0\geq x,\\
e^{-ix\sqrt{\lambda}}[e^{iy\sqrt{\lambda}}-e^{-iy\sqrt{\lambda}}] & 0 \geq y\geq x,
\end{cases}
\end{equation}
There are analogous formulas in the cases $x\geq y\geq 0$, $x\geq 0\geq y$, and $0\geq x\geq y$. We will focus on treating the three cases appearing in \eqref{R1}.

To simplify the presentation below, we will use $\tilde\F f$ to denote quantities that are similar (but not identical) to the Fourier transform of $f$; in particular, we use notation this for quantities that obey the bounds
\begin{equation}\label{FT-bds}
\| \tilde \F f\|_{L^2} \lesssim \|f\|_{L^2},\quad \|\tilde \F f\|_{L^\infty} \lesssim \|f\|_{L^1},\qtq{and}\| |\mu|^{\frac12}\tilde\F f\|_{L^2} \lesssim \|f\|_{H^{\frac12}}. 
\end{equation}
As a typical example, we could apply this notation to a term like
\[
\int_x^\infty e^{i \xi y}f(y)\,dy=(2\pi)^{\frac12}\F[\chi_{(x,\infty)} f](-\xi).
\]
Indeed, the first two bounds in \eqref{FT-bds} can be easily checked (and are uniform in $x$).  The third bound follows from the $L^2$ boundedness of $|\nabla|^{\frac12}\chi_{(x,\infty)}\langle \nabla\rangle^{-\frac12}$, which in turn follows from interpolation:  this is trivial without any derivatives, while the Sobolev embedding $H^1(\R)\subset C_0(\R)$ yields
\begin{align*}
\|\partial_x [\chi_{(x,\infty)}f]\|_{L^2} \lesssim \|\partial_x f\|_{L^2}+|f(x)| \lesssim \|f\|_{H^1}
\end{align*}
uniformly in $x$. 

We begin with the standard $1d$ Strichartz estimates. 

\begin{proposition}[Strichartz estimates]\label{P:Str1}  The following estimates hold on any space-time slab $I\times\R$ with $0\in I$:
\begin{align*}
\| e^{-itH}P_c f
\|_{(L_t^4 L_x^\infty \cap L_t^\infty L_x^2)(I\times\R)} & \lesssim \|f\|_{L^2}, \\
\biggl\| \int_0^t e^{-i(t-s)H}P_c F(s)\,ds
\biggr\|_{(L_t^4 L_x^\infty \cap L_t^\infty L_x^2)(I\times\R)} 
& \lesssim \|F\|_{L_t^\alpha L_x^\beta(I\times\R)}
\end{align*}
for any $(\alpha,\beta)\in[1,\tfrac43]\times[1,2]$ satisfying $\tfrac{2}{\alpha}+\tfrac{1}{\beta}=\tfrac52$. 
\end{proposition}

As is well-known, the proof boils down to the following dispersive estimates.
\begin{lemma}[Dispersive estimates]\label{L:dispersive}  The following estimates hold: 
\[
\|e^{-itH}P_c f\|_{L^2}\lesssim \|f\|_{L^2}
\qtq{and}\|e^{-itH}P_c f\|_{L^\infty} \lesssim |t|^{-\frac12}\|f\|_{L^1}. 
\]
\end{lemma}

\begin{proof}[Proof of Lemma~\ref{L:dispersive}] It is clear that $e^{-itH}P_c$ maps $L^2$ to $L^2$ boundedly.   For the $L^1\to L^\infty$ estimate, we start from \eqref{stone}.  The desired estimate is well-known for the case of the free Schr\"odinger equation, and hence we consider only the contribution of the potential.  After a change of variables, we are left to prove
\[
\sup_x \biggl| \int f(y) \int e^{-it\frac{\lambda^2}{2}
-i(|x|+|y|)\lambda}\tfrac{q}{q-i\lambda}\,d\lambda\,dy\biggr| \leq |t|^{-\frac12} \|f\|_{L^1}. 
\]
We apply Plancherel in the $d\lambda$ integral and observe (by explicit computation) that
\[
\sup_{\theta\in\R} \| \F\bigl(e^{-it\frac{\lambda^2}{2}+i\theta\lambda}\bigr)\|_{L^\infty} \lesssim |t|^{-\frac12}.
\]
Therefore the proof boils down to showing that $\F((q-i\lambda)^{-1})\in L^1$.  In fact, by Cauchy--Schwarz and Plancherel,
\begin{equation}\label{L1}
\|\F((q-i\lambda)^{-1})\|_{L^1} \lesssim \|(1-\partial_{\lambda}^2)(q-i\lambda)^{-1}\|_{L^2} \lesssim 1. 
\end{equation}
The result follows. \end{proof}

We turn to the following weighted estimates for the linear propagator.

\begin{proposition}[Local smoothing estimates]\label{P:Str2} The following estimates hold:
\begin{align}
\|\langle x\rangle^{-\frac32}e^{-itH}P_c f\|_{L_x^\infty L_t^2}&\lesssim \|f\|_{L^2}, \label{P:LS1}\\
\|\partial_x e^{-itH}P_c f\|_{L_x^\infty L_t^2}& \lesssim \|f\|_{H^{\frac12}}.\label{P:LS2}
\end{align}
\end{proposition}

\begin{proof}[Proof of Proposition~\ref{P:Str2}]  We begin by reducing each estimate to one given purely in terms of the resolvent.  Let $m=m(x,\partial_x)\in\{\langle x\rangle^{-\frac32},\partial_x\}$ and $X=L^2$ or $H^{\frac12}$. We will show
\begin{equation}\label{reduce1}
\| m e^{-itH}P_c\|_{X\to L_x^\infty L_t^2} \lesssim \| m E(\lambda)\|_{X\to L_x^\infty L_\lambda^2}. 
\end{equation}
To see this, we let $G\in L_x^1 L_t^2$ and use Plancherel to estimate
 \begin{align*}
\bigl|\langle me^{-itH}P_c f,G\rangle_{L_{t,x}^2} \bigr| & =\biggl| 
\int e^{-it\lambda} \overline{G(t,x)} m(x,\partial_x) E(\lambda;x,y)f(y)\,d\lambda\,dy\,dx\,dt\biggr| \\
& = \biggl|\int \overline{[\F_t^{-1} G](\lambda,x)} [mE(\lambda)f](x)\,dx\,d\lambda\biggr| \\
& \lesssim \|\F_t^{-1} G\|_{L_x^1 L_\lambda^2} \|mE(\lambda)f\|_{L_x^\infty L_\lambda^2} \\
& \lesssim \|G\|_{L_x^1 L_t^2} \|mE(\lambda)\|_{X\to L_x^\infty L_\lambda^2}\|f\|_{X}. 
\end{align*}
Thus \eqref{reduce1} follows. 

Using \eqref{reduce1}, we see that \eqref{P:LS1} will follow from
\begin{equation}\label{E:LS1}
\| \langle x\rangle^{-\frac32} R(\lambda\pm i0) f\|_{L_x^\infty L_\lambda^2} \lesssim \|f\|_{L^2}. 
\end{equation}
We focus on $R(\lambda+i0)$ and write $R=R_1+R_2$ as in \eqref{Res1} and \eqref{Res2}.   The contribution of \eqref{Res2} is easily handled.  In fact, by a change of variables,
\begin{align*}
\biggl\|\tfrac{1}{q-i\sqrt{\lambda}}\int e^{i|y|\sqrt{\lambda}}f(y)\,dy\biggr\|_{L_\lambda^2} ^2 & = \biggl\| \tfrac{1}{q-i\sqrt{\lambda}} \tilde\F f(\sqrt{\lambda})\biggr\|_{L_\lambda^2}^2 \\
& \lesssim \int \frac{|\mu|}{q^2+\mu^2} |\tilde\F f(\mu)|^2\,d\mu \lesssim \|f\|_{L^2}^2.
\end{align*}

To estimate the contribution of \eqref{Res1}, we split into low and high energies.  We let $\chi(\lambda)$ denote a smooth cutoff to $|\lambda|\leq1$ and write $\chi^c=1-\chi$. On the support of $\chi^c$, we can argue as we did for \eqref{Res2}, changing variables and estimating the contribution via
\[
\int_{|\mu|\geq 1}\tfrac{1}{|\mu|}|\tilde\F f(\mu)|^2\,d\mu \lesssim \|f\|_{L^2}^2,
\]
which is acceptable. 

We turn to the low energy contribution of \eqref{Res1}.  Here we use \eqref{R1}; in particular, we will consider the cases $y\geq x\geq 0$ and $0\geq y\geq x$.  In the first case, we use the bound
\[
|e^{ix\sqrt{\lambda}}-e^{-ix\sqrt{\lambda}}| \lesssim |x|\sqrt{\lambda}, 
\] 
and estimate
\begin{align*}
\biggl\| & \langle x\rangle^{-1} \chi(\lambda)\tfrac{1}{\sqrt{\lambda}}\bigl[e^{ix\sqrt{\lambda}}-e^{-ix\sqrt{\lambda}}\bigr]\int_x^\infty e^{i\sqrt{\lambda}y}f(y)\,dy\biggr\|_{L_x^\infty L_\lambda^2} \\
& \lesssim \| \chi(\lambda)\tilde\F f(\sqrt{\lambda})\|_{L_\lambda^2} \lesssim \|\sqrt{\mu}\tilde\F f(\mu)\|_{L_\mu^2(|\mu|\leq 1)}\lesssim \|f\|_{L^2}, 
\end{align*}
which is acceptable.  In the remaining case, we use Cauchy--Schwarz to estimate 
\begin{align*}
\biggl\| & \langle x\rangle^{-\frac32} \chi(\lambda)\tfrac{1}{\sqrt{\lambda}}\int_x^0\bigl[e^{i\sqrt{\lambda}y}-e^{-i\sqrt{\lambda}y}\bigr]f(y)\,dy \biggr\|_{L_x^\infty L_\lambda^2} \\
& \lesssim \|\langle x\rangle^{-\frac32}\chi(\lambda)\int_0^x |y|\,|f(y)|\,dy \biggr\|_{L_x^\infty L_\lambda^2} \lesssim \bigl\|\chi(\lambda) |x|^{\frac32}\langle x\rangle^{-\frac32} \|f\|_{L^2}\bigr\|_{L_x^\infty L_\lambda^2} \lesssim \|f\|_{L^2},
\end{align*}
which is acceptable.  This completes the proof of \eqref{P:LS1}. 

We turn to \eqref{P:LS2}.  Using \eqref{reduce1}, it suffices to prove the following:
\begin{equation}\label{E:LS2}
\|\partial_x R(\lambda\pm i0)f\|_{L_x^\infty L_\lambda^2}\lesssim \|f\|_{H^{\frac12}}. 
\end{equation}
Again we focus on $R(\lambda+i0)$.  Writing $R=R_1+R_2$ as in \eqref{Res1} and \eqref{Res2}, we observe that 
\begin{align*}
\|\partial_x R(\lambda+i0)f\|_{L_x^\infty L_{\lambda}^2} &\lesssim \|a(x,\lambda)\tilde\F f(\sqrt{\lambda})\|_{L_x^\infty L_{\lambda}^2}
\end{align*}
for some bounded function $a$.  Thus the desired estimate follows from a change of variables and \eqref{FT-bds}; indeed,
\[
\|\tilde\F f(\sqrt{\lambda})\|_{L_\lambda^2} \lesssim \| |\mu|^{\frac12}\tilde\F f(\mu)\|_{L_\mu^2} \lesssim \|f\|_{H^{\frac12}}. 
\]
This completes the proof of \eqref{E:LS2} and hence the proof of Proposition~\ref{P:Str2}. \end{proof}

Combining the usual Strichartz estimates (Proposition~\ref{P:Str1}) with the weighted local smoothing estimate in Proposition~\ref{P:Str2} yields the following corollary:

\begin{corollary}\label{C:Str} The following estimate holds: 
\[
\biggl\| \int_0^t e^{-i(t-s)H}P_c F(s)\,ds\biggr\|_{L_t^4 L_x^\infty \cap L_t^\infty L_x^2} \lesssim \|\langle x\rangle^{\frac52} F\|_{L_{t,x}^2}. 
\]
\end{corollary}

\begin{proof} Using the Strichartz estimate Proposition~\ref{P:Str1}, the dual estimate to \eqref{P:LS1}, and Cauchy--Schwarz, we have
\begin{align*}
\biggl\| \int_\R e^{-i(t-s)H}P_cF(s)\,ds\biggr\|_{L_t^4 L_x^\infty\cap L_t^\infty L_x^2} & \lesssim \biggl\|\int_\R e^{isH}P_cF(s)\,ds\biggr\|_{L^2} \\
& \lesssim \|\langle x\rangle^{\frac32} P_cF\|_{L_x^1 L_t^2} \\
& \lesssim \|\langle x\rangle^{\frac52} P_c F\|_{L_{t,x}^2}. 
\end{align*}
The desired estimate now follows from the Christ--Kiselev lemma 
\cite{CK}. \end{proof}

We will also need the following inhomogeneous local smoothing estimates. 

\begin{proposition}\label{P:Str3} For any $t\geq 0$, we have 
\begin{align}
\biggl\| \langle x\rangle^{-1}\int_0^t e^{-i(t-s)H}P_cF(s)\,ds\biggr\|_{L_x^\infty L_t^2} & \lesssim \|\langle x\rangle F\|_{L_x^1 L_t^2}, \label{E:Str31} \\
\biggl\| \int_0^t \partial_x e^{-i(t-s)H}P_cF(s)\,ds\biggr\|_{L_x^\infty L_t^2} & \lesssim \|F\|_{L_x^1 L_t^2}.\label{E:Str32} 
\end{align}
\end{proposition}

\begin{proof}[Proof of Proposition~\ref{P:Str3}]  We begin with the identity
\begin{equation}\label{mid}
\begin{aligned}
2\int_0^t e^{-i(t-s)H}P_cF(s)\,ds & = \int_\R e^{-i(t-s)H}P_cF(s) \,ds \\
& \quad + \int_{0}^\infty e^{-i(t-s)H}P_cF(s)\,ds \\
& \quad -\int_{-\infty}^0 e^{-i(t-s)H}P_cF(s)\,ds. 
\end{aligned}
\end{equation}
In fact, this is a consequence of 
\[
\int_\R e^{-i(t-s)H}P_cF(s)\,ds = \int_{-\infty}^t e^{-i(t-s)H}P_cF(s)\,ds - \int_t^\infty e^{-i(t-s)H}P_cF(s)\,ds,
\]
which follows from the fact that both sides solve 
\[
i\partial_t u = Hu\qtq{with}u(0)=\int_\R e^{isH}P_cF(s)\,ds.
\] 

In light of \eqref{mid}, it therefore suffices to estimate
\[
\int_\R e^{-i(t-s)H}P_c\chi(s)F(s)\,ds,
\]
where $\chi\in\{1,\chi_{(0,\infty)},\chi_{(-\infty,0)}\}$. 

Similar to the proof of Proposition~\ref{P:Str2}, we will use \eqref{stone} and Plancherel to reduce the desired bounds to an estimate given in terms of the resolvent.  In particular, we write
\begin{align*}
\int e^{-i(t-s)H}P_c\chi F(s)\,ds & = \int e^{-it\lambda}E(\lambda)\int e^{is\lambda}\chi(s)F(s)\,ds\,d\lambda \\
& =\F_\lambda\bigl[E(\lambda) \F_s^{-1}(\chi F)\bigr)](t),
\end{align*}
where we use subscripts to denote the variable of integration in the definition of the Fourier transform. Thus, writing $m=m(x,\partial_x)\in\{\langle x\rangle^{-1},\partial_x\}$ and $X=\langle x\rangle^{-1} L^1$ or $X=L^1$, we use Plancherel and Minkowski's inequality to estimate
\begin{align*}
\biggl\| m\int e^{-i(t-s)H}P_c \chi F(s)\,ds\biggr\|_{L_x^\infty L_t^2} & \lesssim \| mE(\lambda)\F_s^{-1}(\chi F)\|_{L_x^\infty L_\lambda^2} \\
& \lesssim \|mE(\lambda)\F_s^{-1}(\chi F)\|_{L_\lambda^2 L_x^\infty} \\
& \lesssim \bigl\| \|mE(\lambda)\|_{X\to L_x^\infty}\|\F_s^{-1}(\chi F)(\lambda)\|_{X} \bigr\|_{L_\lambda^2} \\
& \lesssim \bigl[\sup_{\lambda}\|mE(\lambda)\|_{X\to L_x^\infty}\bigr]\|F\|_{L_t^2 X} \\
& \lesssim \bigl[\sup_{\lambda}\|mE(\lambda)\|_{X\to L_x^\infty}\bigr]\|F\|_{X L_t^2}.
\end{align*}
The proof of \eqref{E:Str31} and \eqref{E:Str32} therefore reduces to the following two estimates:
\begin{align}
\sup_{\lambda} \| \langle x\rangle^{-1} R(\lambda\pm i0) f\|_{L_x^\infty} & \lesssim \|\langle x\rangle f\|_{L_x^1}, \label{31} \\
\sup_{\lambda}\|\partial_x R(\lambda\pm i0)f\|_{L_x^\infty} & \lesssim \|f\|_{L_x^1}.\label{32}
\end{align}

We consider $R(\lambda+i0)$, the other case being similar.  We decompose the kernel as $R_1+R_2$, as in \eqref{Res1} and \eqref{Res2}.  The contribution of $R_2$ to both \eqref{31} and \eqref{32} is handled easily. In fact, since $|q-i\sqrt{\lambda}|\geq |q|$, we have
\[
\|R_2(\lambda) f\|_{L_x^\infty} + \|\partial_x R_2(\lambda)f\|_{L_x^\infty} \lesssim \|\tilde Ff \|_{L_x^\infty} \lesssim \|f\|_{L_x^1}
\]
uniformly in $\lambda$.  

We turn to the contribution of $R_1$.  The contribution to \eqref{32} is straightforward, as we can estimate
\[
\|\partial_x R_1(\lambda)f\|_{L_x^\infty} \lesssim |\tilde\F f(\sqrt{\lambda})| \lesssim \|f\|_{L^1},
\]
uniformly in $\lambda$. For the contribution to \eqref{31}, we recall \eqref{R1}.  In particular, we need only consider the cases $y\geq x\geq 0$ and $0\geq y\geq x$.  In the first case, we estimate
\[
\biggl| \tfrac{1}{\sqrt{\lambda}}(e^{-ix\sqrt{\lambda}}-e^{ix\sqrt{\lambda}})\int_x^\infty e^{i\sqrt{\lambda}y}f(y)\,dy\biggr| \lesssim |x| \tilde\F f(\sqrt{\lambda}),
\]
and hence the desired bound holds in this regime (cf. \eqref{FT-bds}). Finally, if $0\geq y\geq x$, we estimate
\[
\biggl| \tfrac{1}{\sqrt{\lambda}} \int_x^0 [e^{i\sqrt{\lambda}y}-e^{-i\sqrt{\lambda}y}]f(y)\,dy\biggr| \lesssim \| yf(y)\|_{L^1}
\]
uniformly in $x$ and $\lambda$.  Thus the desired bound holds in this regime as well.  This completes the proof of Proposition~\ref{P:Str3}.\end{proof}

Finally, let us record one additional corollary of Proposition~\ref{P:Str2}. 

\begin{corollary}\label{C:Str2} The following estimates hold: 
\begin{align*}
\biggl\| \langle x\rangle^{-\frac32} \int_0^t e^{-i(t-s)H}P_c F(s)\,ds\biggr\|_{L_x^\infty L_t^2(\R\times[0,T])} & \lesssim \|F\|_{L_t^1 L_x^2([0,T]\times\R)}, \\
\biggl\| \partial_x\int_0^t e^{-i(t-s)H}P_cF(s)\,ds\biggr\|_{L_x^\infty L_t^2(\R\times[0,T])} & \lesssim \|F\|_{L_t^1 H_x^{\frac12}([0,T]\times\R)}.
\end{align*}
\end{corollary}

\begin{proof} To rid ourselves of the integral over $[0,t]$ we again use the decomposition \eqref{mid} as in Proposition~\ref{P:Str3} and endeavor to estimate $\chi F$, with 
\[
\chi\in\{1,\chi_{(0,\infty)},\chi_{(-\infty,0)}\}.
\] 
Let $m=m(x,\partial_x)\in\{\langle x\rangle^{-\frac32},\partial_x\}$ and write $X=L^2$ if $m=\langle x\rangle^{-\frac32}$ and $X=H^{\frac12}$ if $m=\partial_x$.  Then using Proposition~\ref{P:Str2}, boundedness of $e^{isH}$ on $X$, and Minkowski's inequality, we may estimate 
\begin{align*}
\biggl\| m\int_\R e^{-i(t-s)H}P_c\chi(s) F(s)\,ds\biggr\|_{L_x^\infty L_t^2} & \lesssim \biggl\|\int_\R e^{-isH} \chi(s)F(s)\,ds\biggr\|_{X} \\
& \lesssim \|F\|_{L_t^1 X}.
\end{align*}
The result follows.
\end{proof}

We close this section with a technical result relating the usual Sobolev spaces with those defined in terms of $H$.  We state the result we need as follows.  In the following, we let $m(\partial_x)$ denote the Fourier multiplier operator with symbol $m(\mu)$. 

\begin{lemma}\label{L:equiv} We have 
\begin{equation}\label{compare-ip}
\langle f, HP_c g\rangle = \langle f,-\tfrac12\partial_x^2 g\rangle + B(f,g),
\end{equation}
where $B(f,g)$ is a linear combination of terms of the form
\[
\langle m(\partial_x)\partial_xf,\partial_x g\rangle,\qtq{where} m(\mu)=(q-i\mu)^{-1}.
\]
%

Consequently, for $f=P_cf$, 
\begin{equation}\label{equiv}
 \|\sqrt{H}f\|_{L^2}\lesssim \|f\|_{\dot H^1}\qtq{and} \| f\|_{\dot H^1} \lesssim \|\sqrt{H}f\|_{L^2}+\|f\|_{L^2}.
\end{equation}
\end{lemma}

Although (\ref{equiv}) has already been shown in \cite[Section VIII,D]{DMW} via the $W^{1,p}$ boundedness of wave operators for $H$, we give a simpler proof  of (\ref{equiv}) by using the explicit representation of $\sqrt{H}$. 

\begin{proof}[Proof of Lemma~\ref{L:equiv}] 
By the spectral theorem and the explicit form of the resolvent, we have the identity
\[
\langle f,HP_c g\rangle = \langle f,-\tfrac12\partial_x^2g\rangle + B(f,g),
\]
where $B(f,g)$ is a linear combination of terms like
\begin{align*}
\iiint \tfrac{\lambda}{\sqrt{\lambda}(q-i\sqrt{\lambda})}e^{i(|x|+|y|)\sqrt{\lambda}} f(x)g(y)\,dx\,dy\,d\lambda = \int \tfrac{\mu^2}{2(q-i\mu)} \tilde\F f(\mu) \tilde \F g(\mu)\,d\mu.
\end{align*}
Here we use the notation
\[
\tilde\F f(\mu)=\int e^{i|x|\mu}f(x)\,dx.
\]
This is consistent with the usage above, and in fact in this case $\tilde\F f$ can be written exactly as the sum of Fourier transforms of $f$ and its reflection.  Thus \eqref{compare-ip} follows from Plancherel.

We turn to \eqref{equiv}. For the first estimate we simply observe that $m(\partial_x)$ maps $L^2\to L^2$ boundedly.  For the second estimate, we observe in fact that $m(\partial_x)\partial_x$ maps $L^2\to L^2$ boundedly, and hence by Young's inequality 
\begin{align*}
\| \partial_x f\|_{L^2}^2 & \lesssim  \|\sqrt{H} f\|_{L^2}^2 + \|f\|_{L^2}\|\partial_x f\|_{L^2} \\
& \lesssim  \|\sqrt{H} f\|_{L^2}^2 + \eps\|\partial_x f\|_{L^2}^2 + \eps^{-1}\|f\|_{L^2}^2
\end{align*}
for any $\eps>0$.  Choosing $\eps\ll 1$ implies the desired bound. 
\end{proof}

\begin{remark}\label{R:L1} The multiplier $m(\partial_x)$ appearing in \eqref{compare-ip} actually maps $L^r\to L^r$ boundedly for any $1\leq r\leq\infty$. Indeed, it was already proven in \eqref{L1} that $\F^{-1}m\in L^1$, and hence this is a consequence of Young's inequality.  In particular, we are not using any multiplier theorems and are able to access the $L^1$, $L^\infty$ endpoints. In a similar way, we see that $m(\partial_x)$ is bounded on 
$L_x^pL_t^q$ for all $1\le p,q\le\infty$. Those will be useful in the proof of Lemma~\ref{L:bs-H1} below. 
\end{remark}

\subsection{Existence of small solitary waves}\label{S:existence}  In this section we discuss the existence and properties of solutions to \eqref{elliptic}.

In \cite{FOO}, the authors considered \eqref{nls} with a focusing nonlinearity and provided an explicit formula for the family of nonlinear bound states.  Using our notation, these solutions are given by
\[
 Q(x) =\left(\tfrac{(p+2)|E|}{2|\mu|}\right)^{\frac{1}{p}} \cosh^{-\frac2p}\left(p\sqrt{\tfrac{|E|}{2}}|x|+\arctanh \left(\tfrac{|q|}{\sqrt{2|E|}}\right)\right),
\]
where $E<-\tfrac12 q^{2}$ and $\mu<0$. This formula is obtained by solving the relevant ODE on each side of $x=0$ and then gluing them together at $x=0$ to impose the jump condition $Q'(0+)-Q'(0-)=2qQ(0)$.  This approach also works in the defocusing case $\mu>0$; the resulting formula is
\[
Q(x) =\left(\tfrac{(p+2)|E|}{2\mu}\right)^{\frac{1}{p}} \sinh^{-\frac2p}
\left(p\sqrt{\tfrac{|E|}{2}}|x|+\arctanh \left(\tfrac{\sqrt{2|E|}}{|q|}\right)\right)
\]
for $-\tfrac12 q^2<E<0$.  When $E=0$, one has the solution 
\[
Q(x) =\left(\tfrac{(p+2)|q|^2}{\mu(p|q||x|+2)^2}\right)^{\frac{1}{p}},
\]
which belongs to $L^2$ provided $p<4$.

From the explicit formulas for $Q$, one can observe that as $E$ approaches $-\tfrac12 q^2$, the functions $Q$ behave like a small multiple of the linear eigenfunction. It will be convenient to describe this behavior in Proposition~\ref{P:Q} below.  In particular, we find it convenient to follow the approach of \cite{GNT} and parametrize the family of ground states by small $z\in \C$. 

In the following, we write $D_jQ[z] = \tfrac{\partial}{\partial z_j}Q[z]$, where we identify $z\in\C$ with the real vector $(z_1,z_2)$.  We write $DQ[z]$ for the Jacobian $DQ[z]:\C\to\C$ with 
\[
DQ[z]w = D_1Q[z]\Re w+ D_2Q[z]\Im w\qtq{for}w\in\C.
\]

We will prove the following. 
\begin{proposition}\label{P:Q} There exists small enough $\delta>0$ such that for $z\in\C$ with $|z|<\delta$, we have the following.
\begin{itemize}
\item There exists a unique solution $Q=Q[z]$ to \eqref{elliptic} with $E=E[|z|]\in\R$. 
\item We may write $Q[z]=z\phi_0+ h$, where
\[
\|h\|_{H^{1,k}(\R)\cap H^2(\R\backslash\{0\})} \lesssim |z|^{2},\quad \|Dh\|_{H^{1,k}} \lesssim |z|,\qtq{and} \|D^2h\|_{L^2} \lesssim 1
\]
for any $k\geq 0$. 
\item $E[|z|]=-\tfrac12 q^2+\mathcal{O}(z)$.
\item $Q[ze^{i\theta}]=Q[z]e^{i\theta}$ and $Q[|z|]$ is real-valued.
\end{itemize}
\end{proposition}

\begin{remark} In fact, the proof will show that $h=\mathcal{O}(z^{p+1})$ and $E[|z|]=-\tfrac12 q^2+\mathcal{O}(z^p)$, but we will not need this refinement in what follows.  Similarly, we can control $Dh$ and $D^2 h$ in the same norms as $h$, but we will not need this. 
\end{remark}

\begin{remark} Using gauge invariance (i.e. differentiating the identity $Q[ze^{i\theta}]=Q[z]e^{i\theta}$) leads to the useful identity
\begin{equation}\label{DQ-id}
Q[z]=-iDQ[z]iz. 
\end{equation}
\end{remark}

Results similar to Proposition~\ref{P:Q} are proved in \cite{SW, SW2}; we will sketch a proof that follows the presentation given in the appendix of \cite{GNT}.  The key ingredient is the following estimate for the resolvent at the linear eigenvalue.

\begin{lemma}\label{L:exist} For any integer $k\geq 0$, $(H+\tfrac12 q^2)^{-1}P_c$ is bounded from $L^2$ to $H^2(\R\backslash\{0\})$ and from $H^{0,k}$ to $H^{1,k}$. 
\end{lemma}

\begin{proof} Evaluating the resolvent at $-\tfrac12 q^2$, we see that the integral kernel of $(H+\tfrac12 q^2)^{-1}P_c$ is a linear combination of terms of the form
\[
e^{|x-y|q}\qtq{and} e^{q(|x|+|y|)}. 
\]
Terms of the second type are straightforward to handle; one needs only observe that 
\[
\biggl| \int e^{q|y|}f(y)\,dy \biggr| \lesssim \|f\|_{L^2}
\]
and that $e^{q|x|}\in H^2(\R\backslash\{0\})\cap H^{1,k}$ for any $k$. It remains to verify that convolution with $e^{q|x|}$ sends $L^2$ to $H^2(\R\backslash\{0\})$ and $H^{0,k}$ to $H^{1,k}$ for any $k$.  Mapping to $H^2(\R\backslash\{0\})$ is clear, so let us consider a weighted norm. As the derivative of $e^{q|x|}$ still decays exponentially, it is enough to work with $H^{0,k}$. The desired estimate therefore reduces to the fact that the operator with kernel $\langle x\rangle^k e^{q|x-y|}\langle y\rangle^{-k}$ maps $L^2\to L^2$ for any $k$ (a consequence of Schur's test, for example).  This completes the proof.  \end{proof}

With Lemma~\ref{L:exist} in place, we turn to the proof of Proposition~\ref{P:Q}. 

\begin{proof}[Proof of Proposition~\ref{P:Q}] We wish to solve 
\[
(H-E)Q+F(Q)=0,\qtq{with} Q=z\phi_0+h\qtq{and} E=-\tfrac12 q^2+e
\]
for small enough (nonzero) $z$,  where $h=\mathcal{O}(z^2)$ is orthogonal to $\phi_0$ and $e=\mathcal{O}(z)$ is real.  Expanding the equation and projecting onto and away from $\phi_0$ leads to the following system for $(e,h)$:
\begin{align}
e &= z^{-1} \langle\phi_0, F(z\phi_0+h)\rangle, \label{syst1}\\
h &= (H+q^2/2)^{-1} \{ -P_cF(z\phi_0+h)+eh\},\label{syst2}
\end{align}
where $z$ is to be small.  To solve this system, let us construct $(e,h)$ as a fixed point of the operator
\[
\Phi(e,h)=(\text{RHS}\eqref{syst1},\text{RHS}\eqref{syst2}).
\]
Let us prove that $\Phi$ is a contraction on the set
\[
A=\{(e,h)\in\R\times P_cH^1:|e|\leq |z|,\quad \|h\|_{H^1} \leq |z|^2\},
\]
where $z$ will be chosen sufficiently small.  We will then prove the desired estimates for $h$ and $e$ as \emph{a priori} estimates using \eqref{syst1} and \eqref{syst2}. 

It is straightforward to show that $\Phi:A\to A$; indeed, writing $(e_1,h_1)=\Phi(e_0,h_0)$ for some $(e_0,h_0)\in A$, we can use Lemma~\ref{L:exist} to estimate
\[
|e_1| \lesssim |z|^{-1}\| z^{p+1}\phi_0^{p+1} + h_0^{p+1}\|_{L^2} \lesssim |z|^p \ll |z|
\]
and
\[
\|h_1\|_{H^1}  \lesssim \|z^{p+1}\phi_0^{p+1}+h_0^{p+1}\|_{L^2}+\|e_0h_0\|_{L^2} \lesssim |z|^{p+1}+|z|^3 \ll |z|^2. 
\]
Similarly, writing $(e_1,h_1)=\Phi(e_0,h_0)$ and $(\tilde e_1,\tilde h_1)=\Phi(\tilde e_0,\tilde h_0)$, we can estimate
\begin{align*}
|e_1-\tilde e_1| & \lesssim |z|^{-1}\|(h_0-\tilde h_0)(z^p\phi_0^p+h_0^p+\tilde h_0^p)\|_{L^2} \\
 & \lesssim |z|^{p-1}\|h_0-\tilde h_0\|_{H^1} \ll \| h_0 - \tilde h_0\|_{H^1}
\end{align*}
and
\begin{align*}
\| h_1-\tilde h_1\|_{H^1} &\lesssim |z|^{p}\|h_0-\tilde h_0\|_{H^1}+|z|^2|e_0-\tilde e_0|+|z|\|h_0-\tilde h_0\|_{H^1} \\
& \ll \|h_0-\tilde h_0\|_{H^1}+|e_0-\tilde e_0|. 
\end{align*}
Thus $\Phi$ defines a contraction on $A$ (for $z$ small enough) and hence has a unique fixed point. 

Using uniqueness and gauge invariance of the nonlinearity, we can deduce that $Q[ze^{i\theta}]=e^{i\theta}Q[z]$ and $E=E[|z|]$.  Similarly, by uniqueness we can guarantee that $Q[|z|]$ is real-valued. 

Next, let us estimate $h$ in $H^2(\R\backslash\{0\})$ and $H^{1,k}$. Using \eqref{syst2}, Lemma~\ref{L:exist}, and Sobolev embedding, we first estimate
\begin{align*}
\|h\|_{H^2(\R\backslash\{0\})} & \lesssim  \|F(z\phi_0+h)+eh\|_{L^2} \\
&  \lesssim \|z^{p+1}\phi_0^{p+1}+h^{p+1}\|_{L^2} +|z| \|h\|_{L^2} \\
& \lesssim |z|^{p+1}+\{|z|^{2p}+ |z|\} \|h\|_{L^2},
\end{align*}
which (for small $z$) implies $\|h\|_{H^2(\R\backslash\{0\})}\lesssim |z|^2$.  Similarly, 
\begin{align*}
\|h\|_{H^{1,k}} & \lesssim |z|^{p+1}+\bigl\{|z|^{2p}+|z|\}\|h\|_{H^{0,k}},
\end{align*}
which again implies $\|h\|_{H^{1,k}}\lesssim |z|^2$. 

To prove bounds for $Dh$, we differentiate \eqref{syst1} and \eqref{syst2}. This leads to
\begin{align}
De &= -z^{-2}Dz\langle \phi_0,F(z\phi_0+h)\rangle + z^{-1}\langle \phi_0, D[F(z\phi_0+h)]\rangle, \label{Dsyst1} \\
Dh &= (H+q^2/2)^{-1}\{-P_c D[F(z\phi_0+h)]+[De]h + e[Dh]\}. \label{Dsyst2}
\end{align}
Using \eqref{Dsyst1}, we can readily deduce that $|De|\lesssim 1$.  Feeding this into \eqref{Dsyst2} and estimating as above using Lemma~\ref{L:exist}, we find
\[
\|Dh\|_{H^{1,k}}  \lesssim |z|^p + |z|\|h\|_{H^{0,k}}+ |z|\|Dh\|_{H^{0,k}} \lesssim |z|^p +|z|^3 + |z|\|Dh\|_{H^{0,k}},
\]
which implies
\[
\|Dh\|_{H^{1,k}} \lesssim |z|,
\]
as desired. Differentiating \eqref{Dsyst1} and \eqref{Dsyst2} once more and arguing similarly yields the final estimate, namely,
\[
\|D^2h \|_{L^2}\lesssim 1. 
\]
This completes the proof. \end{proof}

\subsection{Local well-posedness}\label{S:LWP}

In this section we record a local well-posedness result for \eqref{nls}.  Such results have appeared previously in the literature (e.g. in \cite[Proposition 1]{FOO}); we provide a proof here for the sake of completeness. 

\begin{proposition}[Local well-posedness]\label{P:LWP} For any $u_0\in H^1$, there exists a local-in-time solution to \eqref{nls}.  The solution may be extended as long as the $H^1$-norm does not blow up. 
\end{proposition}

\begin{proof} We will look for $u$ decomposed as follows: 
\[
u(t)=v(t)+a(t)\phi_0:=P_cu(t) + \langle\phi_0,u(t)\rangle\phi_0.
\]  
Equation \eqref{nls} then becomes a coupled system for  $(v(t),a(t))$, namely,
\begin{align}
i\partial_t v(t) & = Hv+P_cF(v(s)+a(s)\phi_0), \nonumber\\
i\partial_t a(t) & =-\tfrac12{q^2} a(t) + \langle\phi_0,F(v(t)+a(t)\phi_0)\rangle. \label{lwp2}
\end{align}
Using an integrating factor in \eqref{lwp2}, we may rewrite these as
\begin{align}
v(t) & = e^{-itH}P_c u_0 -i\int_0^t e^{-i(t-s)H}P_c F(v(s)+a(s)\phi_0)\,ds,  \label{lwp3} \\
a(t)  & = e^{i\frac{1}{2}q^2t}a(0)-i\int_0^t e^{i\frac12 q^2(t-s)}\langle\phi_0, F(v(s)+a(s)\phi_0)\rangle\,ds. \label{lwp4}
\end{align}

Defining $\Phi(v,a)=(\text{RHS}\eqref{lwp3},\text{RHS}\eqref{lwp4})$, we will prove that $\Phi$ defines a contraction on a suitable complete metric space.  Writing $M=\|u_0\|_{H^1}$ and letting $T>0$ to be chosen below, we define
\[
B_{T}=\{(v,a):\|v\|_{(L_t^\infty H_x^1\cap L_t^4 L_x^\infty)([0,T]\times\R)} \leq 2CM,\quad \|a\|_{L_t^\infty([0,T])}\leq 2CM\},
\]
where $C$ encodes constants appearing in Strichartz estimates.  In light of \eqref{equiv}, we can freely exchange $\langle\sqrt{H}\rangle$ and $\langle\partial_x\rangle$ in what follows. 

Writing $(w,b)=\Phi(v,a)$ for some $(v,a)\in B_{T}$, we first use Proposition~\ref{P:Str1} to estimate
\begin{align*}
\lefteqn{\|w\|_{(L_t^\infty H_x^1\cap L_t^4 L_x^\infty)([0,T]\times\R)}}\qquad\\
 & \lesssim 
\|u_0\|_{H^1} + \|F(v+a\phi_0)\|_{L_t^1 H_x^1([0,T]\times\R)} \\
& \lesssim M + T\bigl\{ \|v\|_{L_{t,x}^\infty([0,T]\times\R)}^p 
\|v\|_{L_t^\infty H_x^1([0,T]\times\R)} 
+ \|a\|_{L_t^\infty([0,T])}^{p+1}\|\phi_0\|_{L_x^\infty}^p\|\phi_0\|_{H_x^1}\bigr\}  \\
& \lesssim M + T M^{p+1},
\end{align*}
while
\begin{align*}
\|b\|_{L_t^\infty([0,T])} &
\leq |a(0)|+\|F(v+a\phi_0)\|_{L_t^1 L_x^2([0,T]\times\R)}\|\phi_0\|_{L_x^2} \\
& \lesssim M + TM^{p+1}.
\end{align*}
Thus, for $T=T(M)$ sufficiently small, $\Phi$ maps $B_{T}$ to $B_{T}$.  Similar estimates show that $\Phi$ is a contraction in the norm
\[
d((v,a),(\tilde v, \tilde a)) = \|v-\tilde v\|_{L_t^\infty L_x^2([0,T]\times\R)} 
+ \|a-\tilde a\|_{L_t^\infty([0,T])}
\]
for $T$ sufficiently small.  The result follows. \end{proof}

\section{Setting up the problem}\label{S:setup}

Suppose $u:[0,T]\times\R\to\C$ is a (small) solution to \eqref{nls}.  We will look for a decomposition of $u$ of the form
\begin{equation}\label{E:decompose}
u(t) = Q[z(t)]+v(t).
\end{equation}
We view $z(t)$ as a small unknown to be specified, with $Q$ a solution to \eqref{elliptic} (cf. Proposition~\ref{P:Q}) and $v(t)$ defined through \eqref{E:decompose}. 

Using \eqref{nls}, \eqref{elliptic}, and \eqref{DQ-id}, any such decomposition would lead to an evolution equation for $v$, namely,
\begin{equation}\label{v-eqn1}
\begin{aligned}
i\partial_t v & = Hv + \N, \\
\N& := F(Q+v)-F(Q) -iDQ(\dot z+iEz),
\end{aligned}
\end{equation}
where we have written $Q=Q[z(t)]$, $E=E[|z(t)|]$, and $\dot z$ denotes the time derivative.  We wish to choose $z(t)$ in such a way that the solution to \eqref{v-eqn1} is well-behaved (and such that $z(t)$ remains small).

To choose $z(t)$ and thereby fix the decomposition \eqref{E:decompose}, we will impose the orthogonality conditions
\begin{equation}\label{orthogonal}
\Im\bigl\langle u-Q[z],D_jQ[z]\bigr\rangle = 0 \qtq{for}j\in\{1,2\}
\end{equation}
for all $t\in[0,T]$.  This condition makes $v=u-Q[z]$ orthogonal to the non-decaying solutions to the linearization of \eqref{nls} around $e^{-iEt}Q[z]$ and  agrees with the condition appearing in \cite{GNT}.  We discuss the motivation for this choice in Remark~\ref{R} below. 

The following lemma tells us that as long as the solution $u(t)$ remains small, it is always possible to choose $z(t)$ such that \eqref{orthogonal} holds; moreover, this choice is unique.

\begin{lemma}\label{L:orthogonal} There exists $\delta>0$ small enough such that if $\|u\|_{H^1}\leq \delta$, then there exists unique $z\in\C$ such that \eqref{orthogonal} holds, with $|z|+\|u-Q[z]\|_{H^1} \lesssim \|u\|_{H^1}$. 
\end{lemma}

\begin{proof}[Proof of Lemma~\ref{L:orthogonal}] The proof is the same as \cite[Lemma~2.3]{GNT}.   The idea is that if we were to choose $v=u-\langle \phi_0,u\rangle \phi_0=P_c u$, then we would not be too far off from satisfying \eqref{orthogonal}.  We can therefore use the inverse function theorem to find $z$ exactly satisfying \eqref{orthogonal}.  This is made precise using Proposition~\ref{P:Q}. We sketch the details. 

Denote $\eps=\|u\|_{H^1}$. Define $f:\R^2\to\R^2$ via
\[
f_j(z) = \Im\bigl\langle u-Q[z],D_jQ[z]\bigr\rangle
\]
for $j=1,2$, and set $z_0=\langle \phi_0,u\rangle$.  Note that $|z_0| \leq \eps$.    A computation using the expansion of $Q[z]$ in Proposition~\ref{P:Q} yields
\[
f(z_0)=\mathcal{O}(\eps^2). 
\]
Similarly (using Proposition~\ref{P:Q}), the Jacobian of the map $z\mapsto f(z)$ is computed by
\begin{equation}\label{Jacobian}
D_j f_k(z) = \Im\langle u-Q[z],D_jD_k Q[z]\rangle+\Im\langle D_jQ, D_kQ\rangle= j-k + \mathcal{O}(\eps+|z|). 
\end{equation}
Therefore, by the inverse function theorem, for $\eps$ small enough we may find unique $z$ such that $f(z)=0$.  The result follows.  \end{proof}

Under the (bootstrap) assumption that $\sup_{t\in[0,T]}\|u(t)\|_{H^1}\leq\delta$ for $\delta$ small enough, we can therefore uniquely decompose $u(t)$ in the form \eqref{E:decompose} such that \eqref{orthogonal} holds for each $t\in[0,T]$.

The evolution equation for $v$ is given by \eqref{v-eqn1}.  To derive the evolution equation for $z$, we differentiate the orthogonality conditions \eqref{orthogonal}.  Recalling \eqref{E:decompose}, \eqref{v-eqn1}, and self-adjointness of $H$, this firstly leads to
\begin{align*}
0&=\Im\biggl[ i\langle v,HD_jQ\rangle+i\langle F(Q+v)-F(Q),D_jQ\rangle \\
& \quad -\langle DQ(\dot z+iEz),D_jQ\rangle + \langle v,D_j DQ\dot z\rangle\biggr]. 
\end{align*}
Differentiating \eqref{elliptic} and observing that \eqref{DQ-id} and \eqref{orthogonal} imply $\Im i\langle v,Q\rangle=0$, we may rewrite
\begin{align*}
\Im i\langle v,HD_jQ\rangle  & = \Im \bigl[ i\langle v,ED_jQ\rangle -i\mu\langle v,D_j(|Q|^pQ)\rangle \bigr] \\
& = \Im\langle v,D_jDQ iEz\rangle-\Im i\mu\langle v,D_j(|Q|^p Q)\rangle,
\end{align*}
where we have used \eqref{DQ-id} again in the final line.  Continuing from above, we arrive at the system 
\begin{equation}\label{ode1}
\begin{aligned}
\Im\langle v& ,D_jDQ(\dot z+iEz)\rangle + \Im \langle D_j Q, DQ(\dot z+iEz)\rangle  \\
& \quad = -\Im i\biggl[\langle F(Q+v)-F(Q),D_j Q\rangle - \langle v, D_j(|Q|^p Q)\rangle\biggr]. 
\end{aligned}
\end{equation}
The inner product on the right-hand side of \eqref{ode1} is of the form $\langle G(v,Q),D_jQ\rangle$, where $G$ is at least quadratic in $v$ (see Section~\ref{S:ODE}).  Identifying $\C$ with $\R^2$, we may write this system in the more compact form
\begin{equation}\label{ODE}
A(\dot z+iEz) = b,
\end{equation}
where $A$ is the $2\times 2$ real matrix with entries
\[
A_{jk}= \Im\langle v,D_jD_k Q\rangle + \Im\langle D_jQ, D_kQ\rangle
\]
and $b\in\R^2$ satisfies $b_j=\text{RHS}\eqref{ode1}$.  Note that $A$ coincides with the Jacobian matrix appearing in \eqref{Jacobian}, and hence $A_{jk}=j-k+\mathcal{O}(\delta+|z|)$. 

\subsection{Summary} We have set up the problem as follows: assuming that we have a sufficiently small solution $u$ to \eqref{nls} on a time interval $[0,T]$, we choose $z(t)$ uniquely such that \eqref{orthogonal} holds for each $t$ (using Lemma~\ref{L:orthogonal}).  Defining $v(t)=u(t)-Q[z(t)]$ (where $Q$ is the solution to \eqref{elliptic} as in Proposition~\ref{P:Q}), we find that $v$ and $z$ solve the coupled system \eqref{v-eqn1} and \eqref{ODE}. 

In the next section we will use these equations to prove bounds for $v$ and $z$.  In particular, this will show that $u$ remains small, which implies that the decomposition for $u$ can be continued for all time.  Furthermore, the bounds we obtain will allow us to complete the proof of the main result, Theorem~\ref{T}.

\begin{remark}\label{R}  Let us discuss in some more detail the orthogonality condition \eqref{orthogonal}.  We begin by considering the linearization of \eqref{nls} around a fixed solitary wave $e^{-iEt}Q$. Identifying $v$ with the real vector $v=(\Re v,\Im v)^t$, we can write the linearized equation the form $v_t=Lv$ for an explicit real matrix of operators $L$. Recalling that $Q$ solves \eqref{elliptic} and employing the identity \eqref{DQ-id}, we can connect the functions $D_jQ$ to this linearized equation. In particular (recalling the identification of $\C$ and $\R^2$), one can compute 
\[
L^t iD_jQ = -\tilde E z_j[z_2(iD_1Q) - z_1(iD_2Q)],
\]
where $L^t$ denotes the transpose and we write $D_j E[|z|]= \tilde E z_j$. One therefore finds that that the pair $\{iD_jQ\}$ spans the generalized null space of $L^t$. The orthogonality condition \eqref{orthogonal} is equivalent to the orthogonality of $v$ (identified with the real vector $(\Re v,\Im v)^t$ to $iD_jQ$ (identified with $(-\Im D_j Q, \Re D_jQ)^t)$; here we use the usual inner product for vectors of $\R$-valued functions, i.e. 
\[
(f_1,f_2)^t\cdot (g_1,g_2)^t = \int f_1 g_1 + \int f_2 g_2.
\]

This condition projects $v$ away from the non-decaying solutions to $\partial_t v = Lv$, as we now explain.  We let $\{w_1,w_2\}$ be a basis for the generalized null space of $L^t$ (denoted by $N$) satisfying $L^t w_1=0$ and $L^t w_2=w_1$.  It is not difficult to check that $N^\perp$ is invariant under the flow $\partial_t v = Lv$.  Similarly, for $v(0)\in N$, we can find a solution to $\partial_t v = Lv$ of the form $v(t)=q_1(t) w_1+q_2(t)w_2$.  In fact, explicit computation reveals that $q_1$ and $q_2$ are linear functions in $t$.  Thus, \eqref{orthogonal} exactly projects $v$ away from the non-decaying solutions of $\partial_t v = Lv$, and hence we expect that the component $v$ should decay. 

At a technical level, the key benefit of imposing \eqref{orthogonal} arises in the computation of the ODE \eqref{ode1} for $\dot z+iEz$.  In particular, imposing \eqref{orthogonal} leads to an ODE for $\dot z+ i Ez$ that contains only quadratic and higher terms in $v$.  This is crucial because to describe the asymptotics of $z$ will require that we estimate $\dot z + iEz$ in $L_t^1$, while we can only hope to estimate $v$ in spaces as low as  $L_t^2$ (through reversed Strichartz estimates). 

In contrast, suppose that we were to impose the natural condition 
\begin{equation}\label{other}
\langle v(t),\phi_0\rangle =0,
\end{equation}
so that $v=P_c v$. This type of condition appears in \cite{PW, Weder} and has the advantage of allowing for Strichartz estimates for $e^{-itH}P_c$ to be applied directly to $v$.  In this case, one would find that the ODE for $z$ contains a term that is \emph{linear} in $v$, and hence we would have no hope of estimating in $L_t^1$. 

On the other hand, as $v\neq P_c v$ under the assumption \eqref{orthogonal}, we cannot apply Strichartz estimates for $e^{-itH}$ directly to $v$.  However, if we recall the decomposition $Q[z]=z\phi_0+\mathcal{O}(z^2)$, then we can see that the condition \eqref{orthogonal} implies $\langle v(t),\phi_0\rangle=\mathcal{O}(z^2)$, which suggests that the portion of $v$ parallel to $\phi_0$ should be small compared to $v$.  In fact, in Lemma~\ref{L:vc} we will prove that we can control $v$ by $P_c v$ in all relevant norms, and hence we will be able to utilize the estimates for $e^{-itH}P_c$ after all. 
\end{remark}

\section{Proof of the main result}\label{S:proof}

We suppose $u$ is a solution to \eqref{nls} satisfying 
\begin{equation}\label{bootstrap}
\sup_{t\in[0,T]}\|u(t)\|_{H^1}\leq \delta
\end{equation} 
for $\delta$ sufficiently small, so that we may decompose 
\[
u(t)=Q[z(t)]+v(t),\qtq{where} \Im\langle v(t),D_jQ[z(t)]\rangle\equiv 0 \qtq{for}j\in\{1,2\}
\]
as outlined in the previous section.  By Lemma~\ref{L:orthogonal}, we also have
\[
\sup_{t\in[0,T]}\bigl\{ |z(t)|+\|v(t)\|_{H^1}\bigr\}\lesssim \sup_{t\in[0,T]}\|u(t)\|_{H^1}\lesssim \delta. 
\]

Our goal is to extend these bounds to $[0,\infty)$ and to describe the asymptotics of $z(t)$ and $v(t)$ as $t\to\infty$.  To accomplish this, we will prove a bootstrap estimate using the following norms, which should all be taken over $[0,T]\times\R$ or $[0,T]$.  We first define
\begin{align}
\|v\|_{X}&:=\|v\|_{L_t^\infty H_x^1\cap L_t^4 L_x^\infty} + \|\langle x\rangle^{-\frac32}v\|_{L_x^\infty L_t^2} + \|\partial_x v\|_{L_x^\infty L_t^2}, \label{X} \\
\|z\|_{Y}&:=\|\dot z+iEz\|_{L_t^1 \cap L_t^2}. \label{Y}
\end{align}
Noting that
\[
|z(t)|=\biggl| z(t)\exp\biggl\{i\int_0^t E[z(s)]\,ds\biggr\}\biggr|,
\]
we observe that
\begin{equation}\label{z-infinity}
\|z\|_{L_t^\infty} \leq |z(0)| + \|z\|_Y. 
\end{equation}

As the equation for $v$ involves $Q[z(t)]$, it will be convenient to introduce notation for norms of $Q$ as well.  In particular, we define
\begin{equation}\label{Z}
\|Q\|_Z:= \| \langle x\rangle^{\frac{5}{2}} Q\|_{L_x^1 L_t^\infty\cap L_{t,x}^\infty}+ \|\partial_x Q\|_{L_{t,x}^\infty\cap L_t^\infty L_x^2}
\end{equation}
and
\begin{equation}\label{W}
\|DQ\|_W:=\|\langle x\rangle DQ\|_{L_x^1 L_t^\infty}+\|\langle x\rangle DQ\|_{L_t^\infty L_x^2}+\|\partial_x DQ\|_{L_t^\infty L_x^2}, 
\end{equation}
where $Q=Q[z(t)]$.  Using Proposition~\ref{P:Q}, we can control these norms as long as $z(t)$ remains sufficiently small.

\begin{lemma}\label{L:bs-Q} If $\|z\|_{L_t^\infty}$ is sufficiently small, then
\[
\|Q\|_Z \lesssim \|z\|_{L_t^\infty} \qtq{and} \|DQ\|_W \lesssim 1. 
\]
\end{lemma}

\begin{proof} We begin with the estimate
\begin{equation}\label{an-embedding}
\| \langle x\rangle^\ell G\|_{L_x^r L_t^\infty} \lesssim \|G\|_{L_t^\infty H_x^{1,k}}\qtq{for any} 1\leq r\leq\infty\qtq{and} k>\ell+\tfrac{1}{r}, 
\end{equation}
which follows from H\"older's inequality and the Sobolev embedding $H^1(\R)\hookrightarrow L^\infty(\R)$. In particular,
\[
\|Q\|_Z \lesssim \|Q[z(t)]\|_{L_t^\infty H_x^{1,k}\cap L_t^\infty H_x^2(\R\backslash\{0\})} \qtq{and}\|DQ\|_W \lesssim \|DQ[z(t)]\|_{L_t^\infty H_x^{1,k}}
\]
for large enough $k$.  Here we only use $H^2(\R\backslash\{0\})$ to control $\partial_xQ$ in $L^\infty$. 

The result now follows from Proposition~\ref{P:Q}; indeed, for $\sup_{t\in[0,T]}|z(t)|$ small enough, we can write 
\[
Q[z(t)]=z(t)\phi_0+h(z(t)),
\]
where $h(z(t))=\mathcal{O}(|z(t)|^2)$ and $Dh(z(t))=\mathcal{O}(|z(t)|)$ in the norms detailed in Proposition~\ref{P:Q}. \end{proof}

\subsection{Estimates for the ODE}\label{S:ODE} We first consider the ODE \eqref{ODE} for $z$, which we recall has the form
\[
A(\dot z+iEz)=b,
\]
with $A_{jk}=j-k+\mathcal{O}(\delta+|z|)$ and
\[
 b_j = -\Im i\biggl[\langle F(Q+v)-F(Q),D_j Q\rangle - \mu\langle v, D_j(|Q|^p Q)\rangle\biggr]
\]
To get the error bound on $A_{jk}$, we use Proposition~\ref{P:Q} (similar to the proof of Lemma~\ref{L:bs-Q}). In particular, $A$ is invertible with uniformly bounded inverse.  

\begin{lemma}\label{L:bs-ODE} The following estimate holds:
\[
\|z\|_Y \lesssim \|DQ\|_W\bigl\{\|v\|_X^2 \|Q\|_Z^{p-1}+\|v\|_X^{p+1}\bigr\}.
\]

\end{lemma}

\begin{proof} We examine the right-hand side of the ODE \eqref{ODE} in a more detail.  First, 
\[
D_j(|Q|^p Q) = \tfrac{p}{2}|Q|^{p-2}Q^2 D_j \bar Q + \tfrac{p+2}{2}|Q|^p D_j Q, 
\]
while
\begin{equation}\label{F-diff}
F(Q+v)-F(Q)=\tfrac{p+2}{2}\mu v\int_0^1 |Q+\theta v|^p\,d\theta + \tfrac{p}{2}\mu\bar v\int_0^1 |Q+\theta v|^{p-2}(Q+\theta v)^2\,d\theta. 
\end{equation}
Thus, we may rewrite
\[
b_j=-\Im i\langle G(v,Q),D_jQ\rangle,
\]
where 
\begin{align*}
G(v,Q)& :=\tfrac{p+2}{2}\mu v\int_0^1\bigl[|Q+\theta v|^p-|Q|^p\bigr]\,d\theta 
\\ 
&\quad + \tfrac{p}{2}\mu\bar v\int_0^1 \bigl[|Q+\theta v|^{p-2}(Q+\theta v)^2-|Q|^{p-2}Q^2\bigr]\,d\theta
\end{align*}
In particular,
\begin{equation}\label{Gvqbd}
|G(v,Q)| = \mathcal{O}(v^2 Q^{p-1}+ v^{p+1}). 
\end{equation}

Using the above together with Proposition~\ref{P:Q} and Sobolev embedding, we may now estimate
\begin{equation}\label{ODE-estimate}
 \begin{aligned}
\| \dot z+iEz \|_{L_t^2} & \lesssim \bigl\| \|v\|_{L_x^\infty}^2\|Q\|_{L_x^\infty}^{p-1} + \|v\|_{L_x^\infty}^{p+1} \bigr\|_{L_t^2}\|DQ\|_{L_t^\infty L_x^1}  \\
& \lesssim \|DQ\|_W\bigl\{\|v\|_{L_t^4 L_x^\infty}^2 \|Q\|_{L_{t,x}^\infty}^{p-1} + \|v\|_{L_t^4 L_x^\infty}^2\|v\|_{L_t^\infty H_x^1}^{p-1}\bigr\} \\
& \lesssim   \|DQ\|_W\bigl\{\|v\|_X^2 \|Q\|_Z^{p-1}+\|v\|_X^{p+1}\bigr\},
\end{aligned}
\end{equation}
which is acceptable.  
 We next estimate the $L_t^1$-norm.  Using \eqref{Gvqbd}, we estimate as follows: 
 \begin{align*}
 \| \dot z+iEz\|_{L_t^1}& \lesssim  \int |A^{-1}\Im \langle G(v,Q),DQ\rangle|\,dt  
 \\
 & \lesssim \| G(v,Q) DQ\|_{L_{t,x}^1} \\
 & \lesssim \| \langle x\rangle ^{-\frac32} v\|_{L_x^\infty L_t^2}^2\|\langle x\rangle^{\frac{3}{p-1}}Q\|_{L_{t,x}^\infty}^{p-1}\|DQ\|_{L_x^1 L_t^\infty} \\
 &\quad + \|v\|_{L_t^4 L_x^\infty}^4 \|v\|_{L_{t,x}^\infty}^{p-3} \|DQ\|_{L_t^\infty L_x^1}  \\
 & \lesssim  \|DQ\|_W\bigl\{\|v\|_X^2 \|Q\|_Z^{p-1}+\|v\|_X^{p+1}\bigr\},
 \end{align*}
 which is acceptable.  This completes the proof. \end{proof}
 
 \subsection{Estimates for the PDE} We next consider the PDE \eqref{v-eqn1} for $v$.  

We will prove the following.
\begin{proposition}\label{P:bs-v} The following estimate holds:
\[
\|v\|_X \lesssim \|v(0)\|_{H^1} + \|z\|_Y\|DQ\|_W+\|v\|_X\|Q\|_Z^p + \|v\|_X^{p+1}.
\]
\end{proposition}
 
The plan is to use Strichartz and local smoothing estimates for $e^{-itH}$.  However, we cannot apply these estimates directly to $v$ because the orthogonality conditions \eqref{orthogonal} do \emph{not} imply that $v$ belongs to the continuous spectral subspace of $H$.  Nonetheless, using Proposition~\ref{P:Q} and \eqref{orthogonal}, we can prove that $v$ can be controlled by $P_c v$.
 
\begin{lemma}\label{L:vc}There exists $\delta>0$ small enough that the following holds:  If $\|z\|_{L_t^\infty}\leq\delta$ and $v\in X$ satisfies the orthogonality condition
\begin{equation}\label{orthogonal2}
\Im\langle v(t),D_jQ[z(t)]\rangle\equiv 0 \qtq{for}j\in\{1,2\}
\end{equation}
(where $Q[z]$ is as in Propostion~\ref{P:Q}), then
\[
\|v\|_X\lesssim \|P_c v\|_X.
\]
Here $X$ is as in \eqref{X} and $P_c$ denotes the projection onto the continuous spectral subspace of $H$. 
\end{lemma} 
 
 \begin{proof} Writing $v=P_c v + \langle\phi_0,v\rangle\phi_0$, we see that it suffices to prove 
 \[
 \|\langle\phi_0,v\rangle\phi_0\|_X\ll \|v\|_X. 
 \]
To this end, we use Proposition~\ref{P:Q} to write $Q[z(t)]=z(t)\phi_0+h(z(t))$, with $h(z)=\mathcal{O}(z^2)$ and $Dh(z)=\mathcal{O}(z)$ in the norms detailed in Proposition~\ref{P:Q}.  As \eqref{orthogonal2} yields
\[
|\langle \phi_0,v(t)\rangle| \lesssim  |\langle Dh,v(t)\rangle|,
\]
we can estimate
\[
\| \langle \phi_0,v\rangle\phi_0\|_X \lesssim \|\langle Dh, v(t)\rangle\|_{L_t^2 \cap L_t^\infty}
\] 
We now claim that
\begin{equation}\label{vcvsv}
\|\langle Dh, v(t)\rangle\|_{L_t^2 \cap L_t^\infty} \lesssim \|z\|_{L_t^\infty} \|v\|_X,
\end{equation}
from which the result follows.  To see this, first note that by the triangle inequality and Minkowski's inequality, we have 
\begin{align*}
\| \langle Dh,v(t)\rangle\|_{L_t^2} &\lesssim \| Dh\,v(t)\|_{L_t^2 L_x^1} \\
& \lesssim \|Dh\,v(t)\|_{L_x^1 L_t^2} \lesssim \|\langle x\rangle^{\frac32} Dh\|_{L_x^1 L_t^\infty}\|\langle x \rangle^{-\frac32} v\|_{L_x^\infty L_t^2}.
\end{align*}
Using \eqref{an-embedding}, we see that this term is acceptable.  Next,
\[
\| \langle Dh,v(t)\rangle\|_{L_t^\infty}  \lesssim \|Dh\|_{L_t^\infty L_x^2}\|v\|_{L_t^\infty L_x^2},
\]
which is acceptable as well.  The result follows.  \end{proof}
 
 Using Lemma~\ref{L:vc}, we see that it suffices to estimate the $X$-norm of $P_cv$.  Applying $P_c$ to \eqref{v-eqn1}, we have
  \[
 i\partial_t P_c v = HP_cv + P_c\,\N,
 \]
 where we recall
 \[
 \N = F(Q+v)-F(Q)-iDQ(\dot z+iEz).
 \]
 In particular, 
 \begin{equation}\label{pcduh}
 P_c v(t) = e^{-itH}P_cv(0) -i\int_0^t e^{-i(t-s)H}P_c\,\N\,ds. 
 \end{equation}

We begin with the linear evolution term. 
 
 \begin{lemma}\label{L:bs-data} The following bound holds:
 \[
 \|e^{-itH}P_c v(0)\|_X \lesssim \|v(0)\|_{H^1}. 
 \]
 \end{lemma}
 
 \begin{proof} Recalling the definition of the $X$-norm in \eqref{X}, we find that the lemma follows from Proposition~\ref{P:Str1}, Proposition~\ref{P:Str2}, and \eqref{equiv}. 
 \end{proof} 
 
 We turn to the Strichartz norms for the inhomogeneous term.
 
 \begin{lemma}\label{L:bs-str} The following bound holds:
 \begin{align*}
 \biggl\|\int_0^t e^{-i(t-s)H}P_c\,\N\,ds\biggr\|_{L_t^\infty L_x^2\cap L_t^4 L_x^\infty}&  \lesssim \|z\|_Y\|DQ\|_W + \|v\|_X \|Q\|_Z^p + \|v\|_X^{p+1}.
 \end{align*}
 \end{lemma}

\begin{proof}  Using Corollary~\ref{C:Str} we first estimate
 \begin{align*}
 \biggl\| \int_0^t e^{-i(t-s)H}P_c[DQ(\dot z+iEz)]\,ds\biggr\|_{L_t^\infty L_x^2 \cap L_t^4 L_x^\infty} & \lesssim \|\langle x\rangle^{\frac52}DQ(\dot z+iEz)\|_{L_{t,x}^2} \\
 & \lesssim \| \langle x\rangle^{\frac52}DQ\|_{L_t^\infty L_x^2} \|\dot z+iEz\|_{L_t^2} \\
 & \lesssim \|DQ\|_W \|z\|_Y,
\end{align*}
 which is acceptable.
 
 Next we write nonlinear term in the form
\begin{equation}\label{F-decomp}
\begin{aligned}
&F(Q+v)-F(Q) = F_1 + F_2 + F_3,\qtq{where} \\
&F_1=\mathcal{O}(vQ^p),\quad F_2=\mathcal{O}(v^2 Q^{p-1}+v^{p}Q),\qtq{and} F_3=\mu |v|^p v.
\end{aligned} 
\end{equation}
Such a decomposition is easily achieved under the assumption that $F(u)=\mu|u|^p u$ with $p$ equal to an even integer greater than or equal to four.

The linear term is handled as follows.  Using Corollary~\ref{C:Str}, we have
\begin{align*}
\biggl\| \int_0^t e^{-i(t-s)H}P_c F_1\,ds\biggr\|_{L_t^\infty L_x^2 \cap L_t^4 L_x^\infty} & \lesssim \| \langle x\rangle^{\frac52} Q^p v\|_{L_{t,x}^2} \\
& \lesssim \| \langle x\rangle^{\frac{4}{p}}Q\|_{L_x^{2p} L_t^\infty}^p \|\langle x\rangle^{-\frac32}v\|_{L_x^\infty L_t^2} \\
& \lesssim \|Q\|_Z^p \|v\|_X,
\end{align*}
which is acceptable.

Next, we use Proposition~\ref{P:Str1} to estimate
\begin{align*}
\biggl\| \int_0^t e^{-i(t-s)H} P_c F_2 \,ds\biggr\|_{L_t^\infty L_x^2 \cap L_t^4 L_x^\infty} & \lesssim \|v^2 Q^{p-1}\|_{L_{t,x}^{\frac65}}+ \|v^{p}Q^2\|_{L_{t,x}^{\frac65}}. 
\end{align*}
Using Minkowski's inequality to control $L_x^\infty L_t^4$ by $L_t^4 L_x^\infty$, we firstly estimate
\begin{align*}
\|v^2 Q^{p-1}\|_{L_{t,x}^{\frac65}} & \lesssim \|\langle x\rangle^{-\frac32}v\|_{L_x^\infty L_t^2}^{\frac43} \|v\|_{L_x^\infty L_t^4}^{\frac23} \|\langle x\rangle^{\frac{2}{p-1}}Q\|_{L_x^{\frac{6(p-1)}{5}}L_t^\infty}^{p-1} \\
& \lesssim \|v\|_X^2\|Q\|_Z^{p-1},
\end{align*}
which (after an application of Young's inequality) is acceptable. The other term is treated similarly:
\begin{align*}
\|v^{p}Q\|_{L_{t,x}^{\frac65}}& \lesssim \|\langle x\rangle^{-\frac32}v\|_{L_x^\infty L_t^2}^{\frac43}\|v\|_{L_t^4 L_x^{\infty}}^{\frac23}\|v\|_{L_{t,x}^\infty}^{p-2}\|\langle x\rangle^2 Q\|_{L_x^{\frac65}L_t^\infty} \\
& \lesssim \|v\|_X^{p}\|Q\|_Z,
\end{align*} 
which is again acceptable after applying Young's inequality. 

Finally, the contribution of the $F_3$ term containing only $v$ is estimated as follows:
The purely nonlinear term: use Proposition~\ref{P:Str1}
\begin{align*}
\biggl\| \int_0^t e^{-i(t-s)H}P_c(|v|^p v)\,ds\biggr\|_{L_t^\infty L_x^2 \cap L_t^4 L_x^\infty} & \lesssim \||v|^p v \|_{L_t^{\frac43}L_x^1} \\
& \lesssim \|v\|_{L_t^4 L_x^\infty}^{3}\|v\|_{L_t^\infty L_x^{p-2}}^{p-2} \\
& \lesssim \|v\|_{L_t^4 L_x^\infty}^3 \|v\|_{L_t^\infty H_x^1}^{p-2} \lesssim \|v\|_X^{p+1}, 
\end{align*} 
which is acceptable.  This completes the proof of Lemma~\ref{L:bs-str}.\end{proof}

We next consider the $L_t^\infty \dot H_x^1$ norm of $v$.  We treat this term by an energy estimate.  We will make use of Lemma~\ref{L:equiv}.

\begin{lemma}\label{L:bs-H1} The following estimate holds uniformly over $t\in[0,T]$:
\begin{align*}
\| P_cv(t)\|_{\dot H^1}^2 \leq \|v(0)\|_{\dot H^1}^2 + \|v\|_X\|z\|_Y\|DQ\|_W + \|v\|_X^2\|Q\|_Z^p+\|v\|_X^{p+2},
\end{align*}
where norms are taken over $[0,t]\times\R$. 
\end{lemma}

\begin{proof} By \eqref{equiv}, we have
\[
\|P_c v(t)\|_{\dot H^1} \lesssim \| \sqrt{H}P_c v(t)\|_{L^2}^2 + \|v(t)\|_{L^2}^2.
\]
As the $L_t^\infty L_x^2$ norm is controlled via Lemma~\ref{L:bs-str}, it suffices to estimate $\sqrt{H}P_c v$. 

To this end, we use the self-adjointness of $H$ and \eqref{v-eqn1} to write
\[
\|\sqrt{H}P_cv(t)\|_{L_x^2}^2 = \|\sqrt{H}P_cv(0)\|_{L^2}^2 + \Im \int_0^t\langle \sqrt{H}P_cv(s),\sqrt{H}P_c \,\N\rangle\,ds,
\]
where
\[
\N = DQ(\dot z+iEz) + F_1+F_2+F_3
\]
as in \eqref{F-decomp}.  In fact, we will split the term $F_2$ (which collects the terms of orders $v^2 Q^{p-1}$ through $v^pQ$) further by writing
\[
F_2 = F_2^1+F_2^2,
\]
where $F_2^1$ collects terms that are linear in $Q$.  We do this so that we can group this term with those appearing in \eqref{energy2} below (rather than \eqref{energy3}).  This is necessary because when the derivative lands on $Q$ we cannot additionally absorb weights in order to produce a $\langle x\rangle^{-\frac32}v$ term in $L_x^\infty L_t^2$; indeed, we only control $\partial_x Q$ in $L_{t,x}^\infty$.  Thus we must put the whole term in $L_t^1 L_x^2$; see \eqref{annoying} below.  

We first observe that by \eqref{equiv}, we have
\[
\|\sqrt{H}P_c v(0)\|_{L^2}^2 \lesssim \|v(0)\|_{\dot H^1}^2,
\]
which is acceptable.

We next use Lemma~\ref{L:equiv} to write
\begin{equation}\label{l27}
\begin{aligned}
\int_0^t \langle \sqrt{H}P_cv(s)\sqrt{H}P_c\N\rangle \,ds  & = \int_0^t \langle \partial_x v(s),\partial_x \N\rangle \,ds \\
&\quad + \int_0^t \langle m(\partial_x)\partial_x v(s),\partial_x \N\rangle \,ds
\end{aligned}
\end{equation}
where $m(\mu)=(q-i\mu)^{-1}$ (up to the addition of similar terms).  We claim that both terms in \eqref{l27} may be controlled by
\begin{align}
&\|\partial_x v\|_{L_x^\infty L_t^2} \| \partial_x(F_1+F_2^2) \|_{L_x^1 L_t^2} \label{energy3} \\
& \quad + \|\partial_x v\|_{L_t^\infty L_x^2} \| \partial_x[DQ(\dot z+iEz)+F_2^1+F_3]\|_{L_t^1 L_x^2}. \label{energy2}
\end{align}
For the first term in \eqref{l27}, this follows directly from H\"older's inequality.  For the second term in \eqref{l27}, we use H\"older's inequality and the fact that $m(\partial_x)$ maps 
$L_{x}^{\infty}L_{t}^{2}\to L_{x}^{\infty}L_{t}^{2}$ 
and $L^2\to L^2$ boundedly (see Remark~\ref{R:L1}). 

We turn to estimating the terms in \eqref{energy2} and \eqref{energy3}.

We begin with \eqref{energy3}.  First, by the chain rule:
\begin{align*}
\| &\partial_x F_1\|_{L_x^1 L_t^2} \\
& \lesssim \|\partial_x v\|_{L_x^\infty L_t^2}\|Q\|_{L_x^p L_t^\infty}^p+\|\langle x\rangle^{-\frac32}v\|_{L_x^\infty L_t^2}\| \langle x\rangle^{\frac{3}{2(p-1)}} Q\|_{L_x^{p-1} L_t^\infty}^{p-1} \|\partial_x Q\|_{L_{t,x}^\infty} \\
& \lesssim \|v\|_X \|Q\|_{Z}^p,
\end{align*}
which is acceptable.

We turn to the intermediate terms in $F_2^2$, which contains terms of the order $v^2Q^{p-1}$ through $v^{p-1}Q^2$.  Applying the chain and product rule and Young's inequality, we are led to estimate four types of terms in $L_x^1 L_t^2$ corresponding to these two extreme cases. When the derivative lands on a copy of $v$, we estimate
\begin{align*}
\|(\partial_x v) v Q^{p-1}\|_{L_x^1 L_t^2} & \lesssim \|\partial_x v\|_{L_x^\infty L_t^2}\|v\|_{L_{t,x}^\infty} \|Q\|_{L_x^{p-1} L_t^\infty}^{p-1}, \\
\|(\partial_x v) v^{p-2} Q^2\|_{L_x^1 L_t^2} &\lesssim \|\partial_x v\|_{L_x^\infty L_t^2} \|v\|_{L_{t,x}^\infty}^{p-1} \|Q\|_{L_x^2 L_t^\infty}^2,
\end{align*}
which are acceptable. When the derivative lands on a copy of $Q$, we instead estimate
\begin{align*}
\|v^2 Q^{p-2}\partial_xQ\|_{L_x^1 L_t^2} & \lesssim \|\langle x\rangle^{-\frac32}v\|_{L_x^\infty L_t^2} \|v\|_{L_{t,x}^\infty} \|\langle x\rangle^{\frac{3}{2(p-2)}}Q\|_{L_x^{p-2} L_t^\infty}^{p-2} \|\partial_x Q\|_{L_{t,x}^\infty}, \\
\|v^{p-1}Q\partial_x Q\|_{L_x^1 L_t^2} & \lesssim \|\langle x\rangle^{-\frac32}v\|_{L_x^\infty L_t^2} \|v\|_{L_{t,x}^\infty}^{p-2} \|\langle x\rangle^{\frac{3}{2}}Q\|_{L_x^1 L_t^\infty} \|\partial_xQ\|_{L_{t,x}^\infty},
\end{align*}
which are acceptable. 

We turn to \eqref{energy2}.  We first have \begin{align*}
\| \partial_x v\|_{L_t^\infty L_x^2} \|\partial_x[DQ(\dot z+iEz)]\|_{L_t^1 L_x^2} & \lesssim \| \partial_x v\|_{L_t^\infty L_x^2} \|\partial_x DQ\|_{L_t^\infty L_x^2} \|\dot z+iEz\|_{L_t^1} \\
& \lesssim \|v\|_X\|z\|_Y \|DQ\|_W,
\end{align*}
which is acceptable. 

Next, we estimate the contribution of $F_2^1$ in \eqref{energy2}, which contains terms that are linear in $Q$.  Distributing the derivative, we are led to estimate the following terms.  First,
\begin{equation}\label{annoying}
\|v^p \partial_x Q\|_{L_t^1 L_x^2} \lesssim \|v\|_{L_t^4 L_x^\infty}^4 \|v\|_{L_{t,x}^\infty}^{p-4} \|\partial_x Q\|_{L_t^\infty L_x^2} \lesssim \|v\|_X^p \|Q\|_Z,
\end{equation}
which is acceptable. Next, 
\begin{align*}
\|v^{p-1}(\partial_x v)Q\|_{L_t^1 L_x^2} & \lesssim \|v\|_{L_t^4 L_x^\infty}^2 \| v^{p-3}(\partial_x v)Q\|_{L_{t,x}^2} \\
& \lesssim \|v\|_{L_t^4 L_x^\infty}^2 \|\partial_x v\|_{L_x^\infty L_t^2} \|v\|_{L_{t,x}^\infty}^{p-3}\|Q\|_{L_x^2 L_t^\infty} \\
& \lesssim \|v\|_X^{p}\|Q\|_Z,
\end{align*}
which is acceptable. 

It remains to estimate the contribution of $F_3$ in \eqref{energy2}.  The purely nonlinear term $F_3=\mu|v|^p v$ is estimated as follows: 
\begin{align*}
\| \partial_x v\|_{L_t^\infty L_x^2}\|\partial_x(|v|^p v)\|_{L_t^1 L_x^2} & \lesssim \|v\|_{L_t^4 L_x^\infty}^4\|\partial_x v\|_{L_t^\infty L_x^2}^2\|v\|_{L_{t,x}^\infty}^{p-4} \lesssim \|v\|_X^{p+2}, 
\end{align*}
which is acceptable.  This completes the proof.
\end{proof}

It remains to estimate the contribution of the inhomogeneous Duhamel term to the $L_x^\infty L_t^2$ components of the $X$-norm (cf. \eqref{X}).  The key ingredients will be Proposition~\ref{P:Str3} and Corollary~\ref{C:Str2}. 

\begin{lemma}\label{L:bs-reverse} The following estimates hold: For $m\in\{\langle x\rangle^{-\frac32},\partial_x\}$, 
\[
\biggl\| m\int_0^t e^{-i(t-s)H}P_c\,\N \,ds \biggr\|_{L_x^\infty L_t^2} \lesssim \|DQ\|_W\|z\|_Y + \|v\|_X \|Q\|_Z^p + \|v\|_X^{p+1}.
\]
\end{lemma}

\begin{proof} We recall that
\[
\N = DQ(\dot z + iEz) + F_1 + F_2 + F_3,
\]
where $F_j$ are as in \eqref{F-decomp}.

We first use Proposition~\ref{P:Str3} to estimate\begin{align*}
\biggl\| m\int_0^t e^{-i(t-s)H}P_c[DQ(\dot z+iEz)]\,ds\biggr\|_{L_x^\infty L_t^2} & \lesssim \| \langle x\rangle DQ(\dot z + iEz)\|_{L_x^1 L_t^2} \\
& \lesssim \| \langle x\rangle DQ\|_{L_x^1 L_t^\infty} \|\dot z+iEz\|_{L_t^2}\\
& \lesssim \|DQ\|_W\|z\|_Y,
\end{align*}
which is acceptable. 

Next, we estimate
\begin{align*}
\biggl\| m\int_0^t e^{-i(t-s)H}P_cF_1\,ds\biggr\|_{L_x^\infty L_t^2} & \lesssim \|\langle x\rangle Q^p v\|_{L_x^1 L_t^2} \\
& \lesssim \| \langle x\rangle^{\frac{5}{2p}}Q\|_{L_x^p L_t^\infty}^p \|\langle x\rangle^{-\frac32}v\|_{L_x^\infty L_t^2} \\
& \lesssim \|Q\|_Z^p \|v\|_X,
\end{align*}
which is acceptable. 

The contribution of $F_2$ is estimated by
\begin{align*}
\biggl\|m\int_0^t e^{-i(t-s)H}P_cF_2\,ds\biggr\|_{L_x^\infty L_t^2} 
& \lesssim \|\langle x\rangle v^2 Q^{p-1}\|_{L_x^1 L_t^2} + \| \langle x\rangle v^p Q\|_{L_x^1 L_t^2} \\
& \lesssim \| \langle x\rangle^{-\frac32}v\|_{L_x^\infty L_t^2} \|v\|_{L_{t,x}^\infty} \| \langle x\rangle^{\frac{5}{2(p-1)}}Q\|_{L_x^{p-1}L_t^\infty}^{p-1}\\ &\quad + \| \langle x\rangle^{-\frac32} v\|_{L_x^\infty L_t^2} \|v\|_{L_{t,x}^\infty}^{p-1} \|\langle x\rangle^{\frac52}Q\|_{L_x^1 L_t^\infty} \\
& \lesssim \|v\|_X^2\|Q\|_Z^{p-1}+\|v\|_X^p\|Q\|_Z,
\end{align*}
which is acceptable (after an application of Young's inequality). 

Finally, we use Corollary~\ref{C:Str2} to estimate
\begin{align*}
\biggl\| m\int_0^t e^{-i(t-s)H}P_c F_3\,ds\biggr\|_{L_x^\infty L_t^2} & \lesssim \| |v|^p v\|_{L_t^1 H_x^{\frac12}} \\
& \lesssim \|v\|_{L_t^4 L_x^\infty}^4 \|v\|_{L_{t,x}^\infty}^{p-4}\|v\|_{L_t^\infty H_x^1} \\
& \lesssim \|v\|_X^{p+1},
\end{align*}
which is acceptable.  This completes the proof of Lemma~\ref{L:bs-reverse}.\end{proof}

Finally, using Lemmas~\ref{L:vc}, \ref{L:bs-data}, \ref{L:bs-str}, \ref{L:bs-H1}, and \ref{L:bs-reverse} we complete the proof of Proposition~\ref{P:bs-v}. 

\subsection{Completing the proof}\label{S:finish} In this section, we first use the estimates of the previous two sections in order to close a bootstrap estimate, which allows us to continue the decomposition of $u$ for all time, as well as to prove the desired properties for $z(t)$ and $v(t)$ and hence complete the proof of Theorem~\ref{T}.

We let $u(t)$ be the solution to \eqref{nls} with initial data $u_0$, where $\|u_0\|_{H^1}=\delta$ for some small $\delta>0$.  By local well-posedness and Lemma~\ref{L:orthogonal}, we can uniquely decompose
\begin{equation}\label{pf-decomp}
u(t)=Q[z(t)]+v(t),\qtq{with}\Im\langle v(t),D_jQ[z(t)]\rangle\equiv 0\qtq{for}j\in\{1,2\},
\end{equation}
at least on some time interval, with $|z(t)|+\|v(t)\|_{H^1} \lesssim \|u(t)\|_{H^1}\lesssim\delta$.  On such an interval, we can now collect the estimates from the previous section.  Collecting Lemma~\ref{L:bs-Q}, \eqref{z-infinity}, Lemma~\ref{L:bs-ODE}, and Proposition~\ref{P:bs-v}, we have the following:
\begin{align}
&\|z\|_{L_t^\infty} \lesssim\delta \implies \|Q\|_Z \lesssim \|z\|_{L_t^\infty}\qtq{and} \|DQ\|_W \lesssim 1, \label{BS1} \\
&\|z\|_{L_t^\infty} \leq |z(0)|+\|z\|_Y, \nonumber\\
&\|z\|_Y \lesssim \|DQ\|_W\bigl\{ \|v\|_X^2 \|Q\|_Z^{p-1}+\|v\|_X^{p+1}\bigr\}, \label{BS2}\\
&\|v\|_X \lesssim \|v(0)\|_{H^1}+\|z\|_Y\|DQ\|_W + \|v\|_X\|Q\|_Z^p + \|v\|_X^{p+1}.\label{BS3}
\end{align}
By a standard bootstrap argument (choosing $\delta$ small), it follows that the bounds
\[
\|u(t)\|_{H^1}\lesssim \delta,\quad \|v\|_X\lesssim\delta,\quad\|z\|_{L_t^\infty}\lesssim\delta,\qtq{and} \|z\|_Y\lesssim\delta^2,
\]
as well as the decomposition \eqref{pf-decomp}, persist for all time.

We turn to establishing the asymptotics $v(t)$ and $z(t)$.

First, we prove scattering in $H^1$ for $v(t)$.  We claim that it suffices to prove scattering for $P_c v(t)$.  Writing $v=P_c v+\langle\phi_0,v\rangle\phi_0$, the claim reduces to proving
\begin{equation}\label{scatter1}
\lim_{t\to\infty}\|\langle \phi_0,v(t)\rangle\phi_0\|_{H^1}=0. 
\end{equation}
\begin{proof}[Proof of \eqref{scatter1}]
Using the orthogonality conditions in \eqref{pf-decomp} and using Proposition~\ref{P:Q} to write $Q[z(t)]=z(t)\phi_0+h(z(t))$ (as in the proof of Lemma~\ref{L:vc}), we find
\[
\| \langle\phi_0,v(t)\rangle\phi_0\|_{H^1} \lesssim \|Dh(z(t))\|_{L_x^{\frac43}} \|v(t)\|_{L_x^4}.
\]
As
\[
\|Dh\|_{L_t^\infty L_x^{\frac43}} \lesssim \|z\|_{L_t^\infty}
\]
it suffices to prove that $\|v(t)\|_{L_x^4} \to 0$ as $t\to\infty$.  To see this, we firstly observe (by interpolation of $L_t^\infty L_x^2$ and $L_t^\infty L_x^4$) that $\|v(t)\|_{L_x^4}^4\in L_t^2$.  We will now show that $\partial_t \|v(t)\|_{L_x^4}^4$ is bounded, which implies the desired result.  Using the equation \eqref{v-eqn1} for $v$ and Lemma~\ref{L:equiv} (writing $Hv=HP_cv -q^2\phi_0\langle \phi_0,v\rangle$), we can firstly estimate
\begin{align*}
\partial_t \|v(t)\|_{L^4}^4 & \lesssim \|v\|_{L_{t,x}^\infty}^2 \|\partial_x v\|_{L_t^\infty L_x^2}^2 + \|v\|_{L_t^\infty L_x^3}^3\|v\|_{L_t^\infty L_x^2} \|\phi_0\|_{L_x^\infty}^2 \\
& \quad + \|v\|_{L_t^\infty L_x^6}^3\|F(Q+v)-F(Q)\|_{L_t^\infty L_x^2} \\
& \quad + \|v\|_{L_t^\infty L_x^6}^3 \|DQ\|_{L_t^\infty L_x^2} \|\dot z+iEz\|_{L_t^\infty}
\end{align*}
uniformly in $t$.  Using the bounds on $v$ and $Q[z]$, we see the proof boils down to controlling $\dot z+iEz$ in $L_t^\infty$.  For this, we go back to the ODE \eqref{ODE} and use the computations at the beginning of Lemma~\ref{L:bs-ODE} to bound 
\begin{align*}
\|\dot z+iEz\|_{L_t^\infty} & \lesssim \|(v^2 Q^{p-1}+v^{p+1})DQ\|_{L_t^\infty L_x^1} \\
& \lesssim \|DQ\|_{L_t^\infty L_x^2} \|v\|_{L_t^\infty L_x^4}^2  \bigl\{\|Q\|_{L_{t,x}^\infty}^{p-1}+\|v\|_{L_{t,x}^\infty}^{p-1}\bigr\}.
\end{align*}
This completes the proof of \eqref{scatter1}.
\end{proof}

It finally remains to prove scattering for $P_c v(t)$. For this we use the Duhamel formula \eqref{pcduh} to show that $\{e^{itH}P_cv(t)\}$ is Cauchy in $H^1$.  Indeed, using the estimates from \eqref{L:bs-str} and Lemma~\ref{L:bs-H1}, we can deduce
\[
\|e^{itH}P_cv(t) - e^{isH}P_c v(s)\|_{H^1} \lesssim \|z\|_Y \|DQ\|_W + \|v\|_X\|Q\|_Z^p+\|v\|_X^{p+1},
\]
where now the norms on the right-hand side are restricted to $(s,t)$ (and not all of the components of the $X$-norm are $L^\infty$ in time). Sending $s,t\to\infty$ yields the claim. 

Finally, we note that $\|\dot z+iEz\|_{L_t^1}\lesssim \delta^2$ yields the desired bounds and asymptotics for $z$.  This completes the proof of Theorem~\ref{T}. 

\end{document}